%% file: main.tex
\documentclass[final,12pt]{alt2024} 

\title[Tight Bounds for Local Glivenko-Cantelli]{Tight Bounds for Local Glivenko-Cantelli}
\usepackage{times}
\usepackage{thm-restate}
\usepackage{amsmath}
\usepackage{amssymb}

\usepackage{bbm}

\renewcommand{\set}[1]{\left\{ #1 \right\}}

\newcommand{\step}[1]{\boldsymbol{\operatorname{step}}_{#1}}
\newcommand{\p}[1][p]{\boldsymbol{#1}}

\newcommand{\h}{\operatorname{h}}

\usepackage[capitalize]{cleveref}

\newcommand{\comment}[1]{}

\altauthor{%
  \Name{Mo\"ise Blanchard}
  \Email{moiseb@mit.edu}\\
  \addr   Massachusetts Institute of Technology
  \AND 
  \Name{Václav Voráček}
  \Email{vaclav.voracek@uni-tuebingen.de}\\
  \addr University of T\"ubingen - T\"ubingen AI center
}

\begin{document}
\include{shortcuts}

\maketitle

\begin{abstract}%
This paper addresses the statistical problem of estimating the infinite-norm deviation from the empirical mean to the distribution mean for high-dimensional distributions on $\{0,1\}^d$, potentially with $d=\infty$. Unlike traditional bounds as in the classical Glivenko-Cantelli theorem, we explore the instance-dependent convergence behavior. For product distributions, we provide the exact non-asymptotic behavior of the expected maximum deviation, revealing various regimes of decay. In particular, these tight bounds demonstrate the necessity of a previously proposed factor for an upper bound, answering a corresponding COLT 2023 open problem \citep{cohen2022local,cohen2023open}. We also consider general distributions on $\{0,1\}^d$ and provide the tightest possible bounds for the maximum deviation of the empirical mean given only the mean statistic. Along the way, we prove a localized version of the Dvoretzky–Kiefer–Wolfowitz inequality. Additionally, we present some results for two other cases, one where the deviation is measured in some $q$-norm, and the other where the distribution is supported on a continuous domain $[0,1]^d$, and also provide some high-probability bounds for the maximum deviation in the independent Bernoulli case.
\end{abstract}


\section{Introduction}

We consider the fundamental statistical problem of estimating the maximal empirical mean deviation for multiple independent Bernoulli random variables. Precisely, for a potentially infinite sequence $\p$ of parameters $p(j)\in[0,1]$ for $j\geq 1$, we consider the product distribution $\mu$ such that the coordinates of $X\sim \mu$ are independent Bernoulli random variables with parameters given by $\p$;
that is $\Ebb[\mu] = \p$ (we refer to the textbook \citet{kallenberg1997foundations} for measure-theoretic concerns). Given $n$ i.i.d.\ samples $X_1,\ldots,X_n$ of $\mu$, we aim to understand the maximum deviation of the empirical mean $\hat \p_n = \frac{1}{n}\sum_{i=1}^n X_i$ to the mean $\p$. We mainly focus on its expectation
\begin{equation*}
    \Delta_n(\p) := \Ebb\|\hat \p_n-\p\|_\infty = \Ebb\sup_j \abs{\hat p_n(j) -p(j)},
\end{equation*}

Understanding the convergence of the empirical mean of i.i.d.\ sequences and studying mean estimators are foundational problems in statistical analysis. A substantial body of literature has explored convergence rates for mean estimation problems in fixed dimensions $d$, under diverse distributional assumptions \citep{catoni2012challenging,devroye2016sub,lugosi2019sub,lugosi2019mean,cherapanamjeri2019fast,diakonikolas2020outlier, hopkins2020mean,lugosi2021robust,cherapanamjeri2022optimal, lee2022optimal}. We note that one uses different estimators of the expectation based on the different use cases. For instance, if one seeks for an estimator such that the sample complexity of $\Pbb(\abs{\mu - \hat \mu_n} \geq \varepsilon) \leq \gamma$ is minimized, then sample mean is usually not the right choice.

\sloppy On the other hand, the classical Glivenko-Cantelli theorem provides distribution-free convergence bounds for the empirical mean, quantified by Dvoretzky–Kiefer–Wolfowitz inequality: $\Delta_n(\p)\lesssim \sqrt{\ln(d+1)/n}$ for $d$-dimensional distributions $\mu$. While this rate is optimal up to constants without further assumptions on $p$---it is attained when $\p$ is a $d$-dimensional constant vector $(c,c,\dots,c)$ for some constant $c>0$---this worst-case bound may not capture the correct behavior of $\Delta_n(\p)$ for specific instances of $\p$. In particular, this bound is overly pessimistic when the coordinates of $\p$ decay to 0 sufficiently fast. As a simple example, in the infinite-dimensional case when $p(j)=1/j$ for $j\geq 1$, $\Delta_n(\p)$ converges to $0$ as the number of samples $n$ grows, while the Glivenko-Cantelli theorem does not provide a useful bound. Instead, we are interested in the instance-dependent convergence behavior, which allows us to provide dimension-free results; that is, results without an explicit dependency on dimension, and have potentially an infinite vector $\p$ of non-zero probabilities.

This problem was first posed and studied by~\citet{thomas2018uniform,cohen2022local}. By symmetry, without loss of generality, we will assume that $p(j)\in[0,\frac{1}{2}]$ for every $j\geq 1$, and that the probabilities $p(1),p(2),\ldots$ are sorted in descending order that is $p(j)\geq p(j+1)$ for $j\geq1$. Following the notation of \cite{cohen2022local}, we denote by $[0,\frac{1}{2}]^\Nbb_{\downarrow 0}$ these sequences. Having introduced the following functionals,

\begin{equation*}
    S(\p) = \sup_{j\in\Nbb}p(j)\ln(j+1) \quad \text{and} \quad   T(\p) = \sup_{j\in\Nbb}\frac{\ln(j+1)}{\ln(1/p(j))},
\end{equation*}
they showed that $\Delta_n(\p)$ converges to $0$ if and only if $T(\p)<\infty$.
Further, they characterized the asymptotic behavior of $\Delta_n(\p)$ for sequences for which $T(\p)<\infty$ and showed that it decays as $\sqrt{S(\p)/n}$ for $n\to\infty$, which corresponds to a sub-Gaussian decay regime for binomials, thoroughly studied in the literature \citep{kearns2013large,berend2013concentration,buldygin2013sub}. The same $\sqrt{S(\p)/n}$ behavior also typically arises in the literature on minimax testing and goodness-of-fit problems with Gaussian, multinomial, or Poisson models \cite{valiant2017automatic,balakrishnan2019hypothesis,chhor2020sharp,chhor2021goodness,chhor2022sparse}. In terms of non-asymptotic results, they provide the following upper bound for a universal constant $c>0$,
\begin{equation}\label{eq:bound_cohen}
    \Delta_n(\p) \leq c\left(\sqrt{\frac{S(\p)}{n}} +\frac{T(\p) \ln(n)}{n}\right),\quad n\geq e^3.
\end{equation}
and conjectured that the $\ln(n)$ factor is superfluous in an open problem presented at COLT 2023~\citep{cohen2023open,cohen2022local}. In this work, we completely characterize the non-asymptotic behavior of $\Delta_n(\p)$. In particular, we show that the $\ln(n)$ factor in Eq~\eqref{eq:bound_cohen} is necessary when considering only the functionals $S(\p)$ and $T(\p)$. Our characterization unveils different regimes of decay for $\Delta_n(\p)$, ranging from a
somewhat Poissonian
sub-gamma regime \citep{boucheron2013concentration} to the asymptotic $\sqrt{S(\p)/n}$ sub-Gaussian regime.

\paragraph{Notation} We use the following notations for maxima and minima $a \lor b :=  \max(a,b)$ and $a \land b := \min(a,b)$ respectively. We write $f(n) \gtrsim g(n)$ (respectively $f(n)\asymp g(n)$) when there exists a universal constant $c>0$ (respectively exist universal constants $c,C>0$) such that $f(n)\geq c g(n)$ (respectively $c g(n) \leq f(n) \leq C g(n)$) for every integer $n \geq  1$. The positive part of $x$ is denoted as $[x]_+ :=  \max(0, x)$. Sequences are typed in bold (for example $\p$).

\paragraph{Outline of the paper} We state our main results in \cref{sec:main_results}. We then give an overview of the proof for the characterization of the expected maximum empirical deviation for product distributions in \cref{sec:characterization} and compare to the literature in Subsection~\ref{ssec:comparison}. We next consider general distributions on $\{0,1\}^d$ in \cref{sec:correlated_case}. Last, we discuss in \cref{sec:open_problem} the implications of our results for COLT 2023 open problem and conclude in \cref{sec:conclusion}. Full proofs are given in the appendix.

\section{Main results and discussion}
\label{sec:main_results}

In this section, we outline the main results. In particular, in Subsection~\ref{ssec:independen_bernoulli} we outline the results for the case of independent Bernoulli random variables for $\infty$ norm deviations. Next, in Subsection~\ref{ssec:dependen_bernoulli} we show the results for the case of dependent Bernoulli random variables. As a stepping stone, we derive a variance dependent version of Dvoretzky-Kiefer-Wolfowitz inequality. In Subsection~\ref{ssec:general_distributions} we present the results for the case where the continuous distributions is supported on $[0,1]^d$ and not just on $\{0,1\}^d$. In Subsection~\ref{ssec:lq_norms} we provide the treatment for the case where we have independent Bernoulli random variables, but we measure the deviation in $q-$norm instead of $\infty$ norm. Finally, in Subsection~\ref{ssec:hp} we present some high-probability bounds for the independent Bernoulli case.

\subsection{Non-asymptotic bounds for independent Bernoulli random variables}\label{ssec:independen_bernoulli}

We start with a brief overview of the results, ignoring the corner cases. It turns out that the crucial aspect is to determine the behavior of $\Delta$ for "step-like" sequences of probabilities $\step{J,q}$ such that $\step{J,q}(i)=q$ for all $i\leq J$ and $\step{J,q}(i)=0$ otherwise. 
On the one hand, we will demonstrate that for a given sequence $\p$ it holds that $\Delta(\p) \gtrsim \Delta(\step{i, p(i)})$. The reason is that $\p \geq \step{i, p(i)}$ element-wise; and therefore the random variables following $\p$ have heavier tails. On the other hand, it also holds that $\Delta(\p) \lesssim \sup_{i \geq 1} \Delta(\step{i, p(i)})$, which we will prove through tail-summation. We refer to Figure~\ref{fig:enter-label} for the illustration of this approach.

\begin{figure}[h]
    \centering
    \includegraphics[width=1\textwidth]{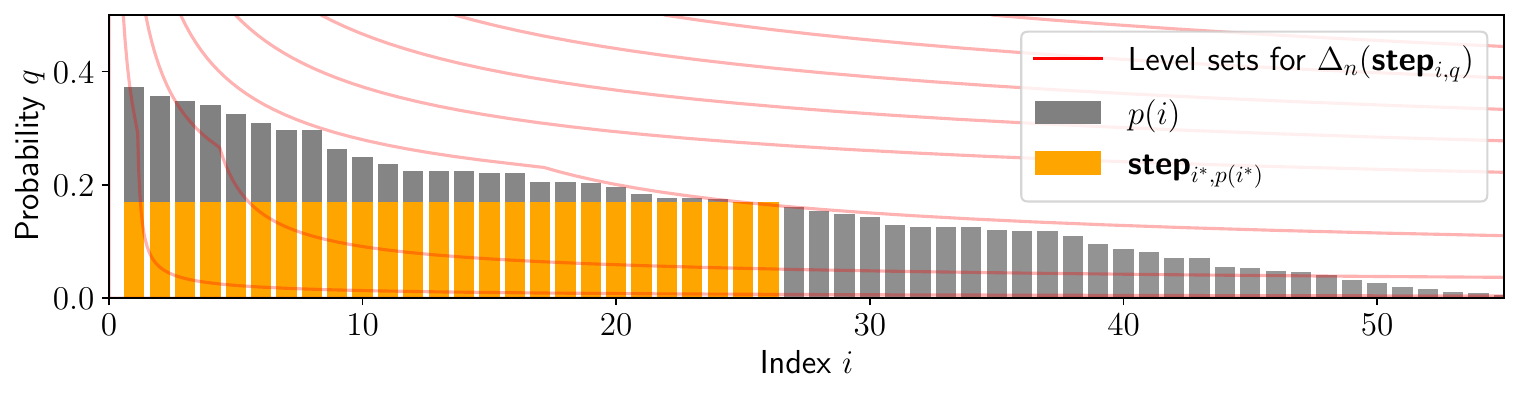}
    \caption{ Illustration of the reduction from general probability profiles $\p$ to step functions $\step{i,q}$. Red curves represent level sets of the expected maximum deviations for step functions $(i,q)\mapsto \Delta_n(\step{i,q})$. The expected maximum deviation $\Delta_n(\p)$ for general probabilities $\p$ is dominated by the maximum deviation of a step function of the form $\step{i,p(i)}$, attained for $i^\star$.}
    \label{fig:enter-label}
\end{figure}

Next, we compute the value of $\Delta(\step{J,q})$ which exhibits three regimes. We state our main characterization in terms of a functional $\phi_{J,q}(n)\asymp\Delta_n(\step{J,q})$.  Formally, $\phi_{J,q}(n)$ is defined for all $n,J\geq 1$ and $q\in[0,\frac{1}{2}]$ via 

\begin{equation}
\label{eq:definition_phi}
     \phi_{J,q}(n) := \begin{cases}
            1 & n\leq \frac{\ln(J+1)}{\ln\frac{1}{q}},\\
            \frac{\ln (J+1)}{n\ln\frac{\ln (J+1)}{nq}} &   \frac{\ln(J+1)}{\ln\frac{1}{q}} \leq n\leq \frac{\ln (J+1)}{eq},\\
            \sqrt{\frac{q\ln (J+1)}{n}} & n\geq \frac{\ln (J+1)}{eq}.
            \end{cases}
\end{equation}
By convention, when $q=0$, we pose $\phi_{J,q}(n)=0$ for all $n,J\geq 1$. We now give some interpretation. 
\begin{itemize}
    \item First, a constant regime when $n\leq T(\step{J,q})$, which was to be expected from the following bound from \cite{cohen2022local},
\begin{equation*}
    \Delta_n(\step{J,q})\geq 1\land \frac{T(\step{J,q})}{n}.
\end{equation*}
    \item The second regime interpolates between a behavior $T(\step{J,q})/n$ when $n\leq \ln(J+1)/q^a$ for some arbitrary (but fixed) exponent $a<1$; and a decay of the form $\ln(J+1)/n$ towards the end of the regime, when $n\sim \ln(J+1)/q$.
    \item Last, the third regime in which $\Delta_n(\step{J,q})\asymp \sqrt{S(\step{J,q})/n}$ specifies when the asymptotic bound from \cite{cohen2022local} is tight.
\end{itemize}


The complete characterization of $\Delta_n(\p)$ additionally exhibits a separate behavior for the small probability regime.
The main result now can be written as follows. 

\begin{restatable}{theorem}{maintheorem}
\label{thm:main_result}
    Let $n\geq 1$ and $\p\in[0,\frac{1}{2}]^\Nbb_{\downarrow 0}$.
    \begin{itemize}
        \item   If for all $j\geq 1$, one has $p(j)\leq \frac{1}{2nj}$, then $\Delta_n(\p)\asymp \frac{1}{n} \land \sum_{j\geq 1}p(j)$.
        \item Otherwise,
        \begin{equation*}
            \Delta_n(\p)\asymp \sup_{j\geq 1}\phi_{j,p(j)}(n)\asymp 1\land \sup_{j\geq 1}\paren{\sqrt{\frac{p(j)\ln(j+1)}{n}} \lor  \frac{\ln(j+1)}{n\ln\paren{2+\frac{\ln(j+1)}{np(j)}}}  }.
        \end{equation*}
    \end{itemize}
\end{restatable}

In the second case, our bounds exhibit the asymptotic sub-gaussian term $  \sqrt{S(\p)/n}$, with a 
sub-gamma
extra term that interpolates between the regime $n\leq T(\p)$ for which $\Delta_n(\p)=\Theta(1)$ and the regime when the sub-gaussian term dominates. As a comparison to the bound Eq~\eqref{eq:bound_cohen} written in terms of the functionals $S(\p)$ and $T(\p)$, in this intermediate regime, the expected maximum deviation lies between $T(\p)/n$ and $T(\p)\ln n/n$. We refer to the end of \cref{sec:characterization} for a complete discussion on the implications of this result.

As a consequence of the characterization, we answer the open problem \citep{cohen2023open} by the negative. We show that if one only seeks bounds of $\Delta_n(\p)$ in terms of the sub-Gaussian term $\sqrt{S(\p)/n}$, and the functional $T(\p)$, there are sequence instances for which the $\ln(n)$ term from Eq~\eqref{eq:bound_cohen} is necessary. A constructive proof can be found in \cref{sec:open_problem}.

\begin{restatable}{theorem}{thmopenquestion}
\label{thm:open_problem}
    Suppose that there exists a constant $C\geq  1$ and $n_0\geq 1$, and a function $\psi:\Nbb\to\Rbb$ such that the inequality  
\begin{equation*}
    \Delta_n(\p) \leq C\sqrt{\frac{S(\p)}{n}} + \frac{T(\p)}{n}\psi(n)
\end{equation*}
holds for all $n\geq n_0$ and $\p\in[0,\frac{1}{2}]^\Nbb_{\downarrow 0}$ (product measures), then for an integer $n_1$ and a constant $c>0$ depending only on $C$,
\begin{equation*}
    \psi(n) \geq c \ln n,\quad n\geq n_0\lor n_1.
\end{equation*}
\end{restatable}

\subsection{Non-asymptotic bounds for correlated Bernoulli random variables}\label{ssec:dependen_bernoulli}

The previous results focused on the particular case of product measures $\mu $ on $\{0,1\}^\Nbb$, i.e., such that all coordinates of $X\sim\mu$ are mutually independent. In that case, the mean $\p=\Ebb_\mu[X]$ completely characterizes the distribution, which in turn allows having the precise descriptions of the decay rate of $\Delta_n(\p)$ from \cref{thm:main_result}. Similarly, one can consider the considerably more general case of arbitrary distributions $\mu$ on $\{0,1\}^\Nbb$ when coordinates may be correlated. As before, we study the expected maximum deviation $\Delta_n(\mu)=\Ebb \|\hat \p_n - \p\|_\infty$, where $\hat \p_n = \frac{1}{n}\sum_{i=1}^n X_i$ for i.i.d.\ samples $X_i\sim \mu$. Our upper bounds from \cref{thm:main_result} extend directly to the general case; however, these may not be tight in general.

\begin{corollary}
    \label{cor:correlated_case}
    Let $\mu$ be a distribution on $\{0,1\}^\Nbb$ with mean $\p=\Ebb_{X\sim \mu}[X]$. Without loss of generality, suppose that $\p\in[0,\frac{1}{2}]^\Nbb_{\downarrow 0}$.
    \begin{itemize}
        \item   If for all $j\geq 1$, one has $p(j)\leq \frac{1}{2nj}$, then $p(1)\lesssim \Delta_n(\mu)\lesssim \frac{1}{n} \land \sum_{j\geq 1}p(j)$.
        \item Otherwise,
        \begin{equation*}
            p(1)\land \sqrt{\frac{p(1)}{n}}\lesssim \Delta_n(\mu)\lesssim 1\land \sup_{j\geq 1}\paren{\sqrt{\frac{p(j)\ln(j+1)}{n}} \lor  \frac{\ln(j+1)}{n\ln\paren{2+\frac{\ln(j+1)}{np(j)}}}  }.
        \end{equation*}
    \end{itemize}
\end{corollary}

We emphasize that the gap between the upper and lower bounds from \cref{cor:correlated_case} can be large in general, but we show in \cref{sec:correlated_case} that these are the tightest bounds achievable if one only uses the mean statistic $\p$ to describe the distribution $\mu$. As an extreme example, if $p(1)=p(i)$ for all $i\geq 1$, we can consider the perfectly-correlated case when $\mu$ is such that for $X\sim\mu$, almost surely $X(j)=X(1)$ for all $j\geq 1$. In that case, understanding $\Delta_n(\mu)$ reduces to computing the deviation from the mean for a single binomial $Y\sim \Bcal(n,q)$ where $q=p(1)\in[0,\frac{1}{2}]$. It is well known that in this case, $\Ebb|Y-nq| \asymp nq\land \sqrt{nq}$ (e.g. \cite{berend2013sharp}), which corresponds to the lower bounds provided in Corollary \ref{cor:correlated_case}.

En route to proving the tightness of Corollary \ref{cor:correlated_case} for bounds involving only the mean statistic $\p$, we prove a localized version of the classical Dvoretzky-Kiefer-Wolfowitz (DKW) inequality \citep{massart1990tight} which is of independent interest. Given $n$ i.i.d.\ samples $X_1,\ldots, X_n$, from a real-valued random variable $X$, let $F:x\mapsto \Pbb(X\leq x)$ be the cumulative distribution function (CDF) of $X$, and let $F_n:x\mapsto \frac{1}{n} \sum_{i=1}^n \1(X_i\leq x)$ be the empirical CDF. The standard DKW theorem shows that the deviations of $F_n(x)$ can be bounded uniformly in $x$.

\begin{theorem}[DKW theorem \citep{massart1990tight}]
\label{thm:dkw}
    Let $X_1,\ldots,X_n$ be i.i.d.\ samples and denote by $F$ (respectively $F_n$) the true CDF (respectively empirical CDF). Then, for any $t\geq 0$,
    \begin{equation*}
        \Pbb\paren{\sup_{x\in\Rbb} |F_n(x) - F(x)| > \frac{t}{\sqrt n} } \leq 2e^{-2 t^2}.
    \end{equation*}
\end{theorem}

We aim to bound the deviation of the CDF on a smaller interval $[x_0,x_1]$ instead of the full domain $\Rbb$. Indeed, when the maximum variance of $F(x)$ for $x\in[x_0,x_1]$ is small, one would expect to have stronger empirical deviation bounds than those provided by the vanilla \cref{thm:dkw}. We note that \cite{maillard2021local} provides an exact formula for the localized deviation of the CDF. This can be computed numerically with the formula, but an analytical simple upper bound will be more convenient for our purposes. 
In the following result, we show that one can achieve essentially the same DKW tail bounds uniformly on the interval $[x_0,x_1]$ as those for the single random variable $F_n(x)$ for $x\in[x_0,x_1]$ that has maximum variance. The proof is deferred to \cref{sec:proof_local_dkw}.

\begin{theorem}\label{thm:local_dkw}
    Let $X_1,\ldots,X_n$ be i.i.d.\ samples and denote by $F$ (respectively $F_n$) the true CDF (respectively empirical CDF). Then, for any $x_0\leq x_1\in\Rbb\cup\{\pm\infty\}$ and $t\geq 0$, if $V = \max_{x\in[x_0,x_1]} F(x)(1-F(x))$ (with the convention $F(-\infty)=0$ and $F(+\infty)=1$), we have
    \begin{equation*}
        \Pbb\paren{\sup_{x\in [x_0,x_1]} |F_n(x) - F(x)| > t \sqrt{\frac{V}{n}} } \leq c_1 e^{-c_2 \min (t^2,t\sqrt{nV})},
    \end{equation*}
    for some universal constants $c_1,c_2>0$.
\end{theorem}
A similar result recently appeared in~\citet{bartl2023variance} which gives a variance-dependent DKW inequality. They show that for some absolute constants $c,c'>0$ and any $t \geq c\sqrt{\ln\ln n}$,
\begin{equation*}
\Pbb\paren{\exists x\in I_t \text{ s.t. } \abs{F_n(x) - F(x)} > t\sqrt\frac{F(x)(1-F(x))}{n}} \leq 2e^{-c't^2},
\end{equation*}
where $I_t = \{x \mid   t\leq \sqrt{nF(x)(1-F(x))}\}$ is precisely the set of points falling in the sub-Gaussian regime in our \cref{thm:local_dkw}. Note that the width of their confidence-band depends on the variance of the empirical CDF at that point. This is in a contrast with our result, where having confidence-band of uniform width allowed us to derive a bound valid for all $t\geq 0$.

\subsection{Non-asymptotic bounds for general distributions on $[0,1]$}\label{ssec:general_distributions}
The results so far focused on the case when the distributions are supported on $\{0,1\}^\mathbb{N}$. However, some of the results can be generalized for the case of $[0,1]^\mathbb{N}$, as detailed below.

\begin{corollary}\label{cor:continuous[0,1]}
    Let $\mu$ be a distribution on $[0,1]^\Nbb$. Let $\sigma^2(i) = \text{Var}_{X\sim\mu}(X_i)$ for $i\geq 1$ be the variance of coordinate $i$. Without loss of generality, suppose that $\mb{\sigma}^2$ is decreasing.
    \begin{itemize}
        \item If for all $j\geq 1$, one has $\sigma^2(j)\leq \frac{1}{2nj}$, then
        \begin{equation*}
           \Delta_n(\mu) \lesssim \frac{1}{n}\land\sqrt{\frac{\sum_{j\geq 1}\sigma^2(j)}{n}}.
        \end{equation*}
        \item Otherwise,
        \begin{equation*}
            \Delta_n(\mu) \lesssim   1\land \sup_{j\geq 1} \paren{\sqrt{\frac{\sigma^2(j)\ln(j+1)}{n}} \lor  \frac{\ln(j+1)}{n\ln\paren{2+\frac{\ln(j+1)}{n\sigma^2(j)}}}  }.
        \end{equation*}
    \end{itemize}
\end{corollary}

The proof is given in \cref{sec:proof_continuous01}.
Given that the random variables are supported on $[0,1]$, we can use the inequality $\sigma^2(i) \leq p(i)$ for all $i\geq 1$ to obtain similar (but weaker) bounds as in Corollary \ref{cor:continuous[0,1]} but replacing the variances $\sigma^2(i)$ by the means $p(i)$. As for the case of distributions on $\{0,1\}^\Nbb$, the upper bounds from the previous result are not tight in general, however, these are the tightest bounds achievable if one only uses the variance statistic $\mb{\sigma}^2$. In particular, the case of independent Bernoulli random variables characterized in \cref{thm:main_result} always achieves the upper bound except in the regime when $\sum_{j\geq 1}\sigma^2(j)\leq \frac{1}{2n}$. In that case, we can show that the upper bound is attained not by random variables supported on $\{0,1\}$, but on $\{0,\sqrt{2n\sum_{j\geq 1}\sigma^2(j)}\}$. We refer to \cref{sec:proof_continuous01} for further details.

\subsection{Expected empirical deviations in $\ell^q$ norms}\label{ssec:lq_norms}

While the infinite norm deviation $\Delta_n(\mu) = \Ebb\|\hat \p_n - \p\|_\infty$ is the main focus of this paper, a natural question is whether we can obtain similar results for general $\ell^q$-norm expected deviations for $q\geq 1$. We have the following characterization for the decay of the expected $\ell^q$ deviation.

\begin{proposition}\label{prop:characterization_lq_convergence}
    Let $\p\in[0,\frac{1}{2}]_{\downarrow 0}^\Nbb$. Then, $\lim_{n\to\infty} \Ebb\|\hat\p_n-\p\|_q=0$ if and only if $\|\p\|_1<\infty$. Moreover, if $\|\p\|_1=\infty$, then $\|\hat\p_n-\p\|_q=\infty\;(a.s.)$.
\end{proposition}
The proof is given in \cref{sec:lq_norms}.
The analysis of the convergence of $\Ebb\|\hat\p_n-\p\|_q$ is quite different from the $\ell^\infty$ case since for instance the quantity $\Ebb \|\hat \p_n - \p\|_q^q$ can be computed directly as a sum of expectations.
In particular, one can obtain bounds on the expected $\ell^q$ deviation $\Ebb\|\hat \p_n - \p\|_q$ using the following Jensen inequalities,
\begin{equation*}
    \paren{\sum_{j\geq 1}(\Ebb|\hat p_n(j)-p(j)|)^q}^{1/q} \leq \Ebb \|\hat \p_n-\p\|_q \leq (\Ebb \|\hat \p_n-\p\|_q^q)^{1/q}.
\end{equation*}
These bounds give the correct asymptotic convergence rate of the expected $\ell^q$ deviation when $q\geq 2$ up to a factor $\Theta(\sqrt q)$.

\begin{proposition}\label{prop:lq_asymptotic}
    Let $\p\in[0,\frac{1}{2}]_{\downarrow 0}^\Nbb$ such that $\|\p\|_1<\infty$, and $q\geq 2$. Then,
    \begin{equation*}
        1\lesssim \liminf_{n\to\infty} \sqrt{\frac{n}{\|\p\|_{q/2}}} \Ebb\|\hat\p_n-\p\|_q \leq \limsup_{n\to\infty} \sqrt{\frac{n}{\|\p\|_{q/2}}} \Ebb\|\hat\p_n-\p\|_q  \lesssim \sqrt q.
    \end{equation*}
\end{proposition}
Hence the convergence in this case is of the order $\sqrt{\frac{\|\p\|_{q/2}}{n}}$. The proof of this result as well as non-asymptotic bounds can be found in \cref{sec:lq_norms}.

\subsection{High probability bounds for independent Bernoulli}\label{ssec:hp}
While we focused on bounding the expectation of the maximal deviation, we also provide some high probability concentration bounds.
From the bounded differences inequality~\cite[Thm. 6.2]{boucheron2013concentration} -- also known as McDiarmid's inequality -- we can directly have for $\gamma\in(0,1)$,
\[
    \Pbb\left(\abs{\|\hat\p_n - \p\|_\infty - \Delta_n(\p)} \geq \sqrt{\frac{\ln\frac{2}{\gamma}}{2n}}\right) \leq \gamma.
\]
Notably, this bound is often pessimistic and can be significantly tightened. To write the high-probability bounds concisely, we extend the definition of the quantities $\phi_{J,q}(n)$ to all reals $J>0$. For $J\geq 1$, we extend the definition with the same formula in Eq~\eqref{eq:definition_phi}.
For $J\in (0,1)$, we pose
\begin{equation*}
    \phi_{J,q}(n) := e^{-1/J}\sqrt{\frac{q}{n}}.
\end{equation*}
We are now ready to state the high-probability bounds.

\begin{proposition}\label{prop:high_probability_bounds}
Let $\gamma\in(0,\frac{1}{2})$ and $\p\in[0,\frac{1}{2}]^\Nbb_{\downarrow 0}$ such that there exists $j\geq 1$ with $p(j)\geq \frac{\gamma}{2nj}$. Then, for some universal constants $a_1, a_2>0$,
\begin{equation*}
	\Pbb\paren{\|\hat \p_n - \p\|_\infty \geq a_1 \sup_{j\geq 1}\phi_{\frac{j}{\gamma},p(j)}(n)} \leq \gamma.
\end{equation*}
Also,
\begin{equation*}
\Pbb\paren{\|\hat \p_n - \p\|_\infty \lor \frac{1}{n} \leq a_2 \sup_{j\geq 1}\phi_{\frac{j}{\ln 1/\gamma},p(j)}(n)} \leq \gamma.
\end{equation*}

Let $\p\in[0,\frac{1}{2}]^\Nbb_{\downarrow 0}$ such that $p(j)\geq \frac{\gamma}{2nj}$ for all $j\geq 1$. Then,
\begin{equation*}
	\Pbb\paren{\sup_{j\geq 1} \hat p_n(j) \geq \frac{2}{n} }\leq \gamma \quad \text{and} \quad 								\Pbb\paren{\sup_{j\geq 1} \hat p_n(j) = \frac{1}{n} }\asymp 1\land n\sum_{j\geq 1}p(j).
\end{equation*}
\end{proposition}

The proof of this result and further high-probability bounds showing that these are tight in most cases can be found in Appendix~\ref{app:high_prob}.

\section{Expected maximum empirical mean deviation for product distributions}
\label{sec:characterization}

In this section, we give the main steps for the proof of our main characterization in \cref{thm:main_result}. For the sake of conciseness, we only present sketches of the proofs here, all formal proofs of this section are given in \cref{sec:proof_characterization}.

\subsection{Preliminaries and general strategy}
We first recall some basic tail inequalities for binomials. In the following, $D(q \parallel p) = q\ln( \frac{q}{p}) + (1-q)\ln(\frac{1-q}{1-p})$ is the KL-divergence between Bernoulli distributions with parameters $p,q\in[0,1]$. We start with the classical Chernoff bound~\citep{boucheron2013concentration}.

\begin{lemma}[Chernoff bound]
\label{lemma:chernoff_bound}
    For any $0\leq p\leq q\leq 1$, letting $Y\sim \Bcal(n,p)$, we have
    \begin{equation*}
        \Pbb\paren{\frac{Y}{n} \geq q}\leq e^{-nD(q \parallel p)}.
    \end{equation*}
\end{lemma}
We will also use the following anti-concentration bound from \citet[Theorem 9]{zhang2020non}.

\begin{lemma}[\cite{zhang2020non}]
\label{lemma:anti_concentration}
    There exist constants $0<c_0<\frac{1}{4}$ and $C\geq 1$, such that for any $0<p\leq q <1$ satisfying  $\frac{1}{n}\leq q \leq \frac{1+p}{2}$, letting $Y\sim\Bcal(n,p)$, we have
    \begin{equation*}
        \Pbb\paren{\frac{Y}{n} \geq q} \geq c_0 e^{-C n D(q\parallel p)}.
    \end{equation*}
\end{lemma}

As a first observation, defining $\Delta_n^+(\p) = \Ebb\sup_{j\geq 1}[\hat p_n(j)-p(j)]_+$ and $\Delta_n^-(\p) =\Ebb\sup_{j\geq 1}[p(j)-\hat p_n(j)]_+$, we have the following decomposition,
\begin{equation*}
    \frac{1}{2}(\Delta_n^+(\p) + \Delta_n^-(\p)) \leq  \Delta_n^+(\p) \lor \Delta_n^-(\p) \leq \Delta_n(\p) \leq  \Delta_n^+(\p) + \Delta_n^-(\p).
\end{equation*}

We will show in the rest of this paper that the leading term is $\Delta_n^+(\p)$, which we now focus on. The main intuition is that Bernoulli random variables $\Bcal(p)$ with $p\leq 1/2$ have heavier right tails than left tails. To give estimates for $\Delta_n^+(\p)$ for general values of the sequence $\p$, we first start with a reduction to the case when the profile of $\p$ is ``step-like''. Consider such a vector $\step{J,q}$ with $p(i)=q$ for all $i\in[J]$ and $p(i)=0$ for $i>J$. Then, 
\begin{equation*}
\Pbb\paren{\max_{i\leq J} \{ \hat p(i) - p(i)\} \geq \varepsilon} = 1-(1-\Pbb\paren{\hat p(1) - p(1) \geq \varepsilon})^J.
\end{equation*}
Intuitively, this probability is approximately $1-\exp(J\Pbb\paren{\hat p(1) - p(1) \geq \varepsilon})$. If $\varepsilon$ is the expected maximal deviation, then one would expect this probability above to be bounded away from both $0$ and $1$ by some absolute constants. This motivates the definition of the following quantity $\varepsilon_{J,q}(n)$ for any $J\geq 1$ and $q\in(0,1/2]$, where $c_0\in(0,\frac{1}{4})$ is the same constant as in Lemma \ref{lemma:anti_concentration},
\begin{equation*}
    \varepsilon_{J,q}(n) = \inf\set{\varepsilon\geq 0: \Pbb_{Y\sim\Bcal(n,q)}\paren{\frac{Y}{n}\geq q+\varepsilon}\leq \frac{c_0}{2J} },
\end{equation*}
In particular, note that $q+\varepsilon_{J,q}(n)\in \{0,\frac{1}{n},\ldots,\frac{n-1}{n},1\}$ and that
\begin{equation*}
    \Pbb_{Y\sim\Bcal(n,q)}\paren{\frac{Y}{n}> q+\varepsilon_{J,q}(n)} \leq \frac{c_0}{2J} < \Pbb_{Y\sim\Bcal(n,q)}\paren{\frac{Y}{n}\geq q+\varepsilon_{J,q}(n)}.
\end{equation*}
Our goal is to give a characterization of $\Delta_n^+(\p)$ using these coefficients.

\subsection{Step-like sequences describe the behavior of general sequences}

It turns out that not only $ \Delta_n^+(\step{J,q}) \asymp \varepsilon_{J,q} $, but we even have $\Delta_n^+(\p) \asymp \sup_{i \geq 1} \varepsilon_{i, p(i)}$ for most vectors $\p\in[0,\frac{1}{2}]^\Nbb_{\downarrow 0}$ as shown in the following result.

\begin{proposition}
\label{prop:reduction}
    Let $\p\in[0,\frac{1}{2}]^\Nbb_{\downarrow 0}$. Suppose that there exists $i\geq 1$ such that $\varepsilon_{i,p(i)}(n) \geq 0$. 
    Then, there exist universal constants $c,C>0$ such that for all $n\geq 2$,
    \begin{equation*}
        c \cdot \sup_{i\geq 1} \varepsilon_{i,p(i)}(n)  \leq \Delta_n^+(\p) \leq C \cdot \sup_{i\geq 1} \varepsilon_{i,p(i)}(n).
    \end{equation*}
    Further, the upper bound holds for any general distribution $\mu$ on $\{0,1\}^\Nbb$, that is, with $\mb p = \Ebb_{X\sim\mu}[X]$,
    \begin{equation*}
        \Ebb_{X_i\overset{i.i.d.}{\sim}\mu}\sup_{i\in\Nbb}[\hat p_n(i) - p(i)]_+ \leq C \cdot \sup_{i\geq 1} \varepsilon_{i,p(i)}(n).
    \end{equation*}
\end{proposition}

\paragraph{Sketch of proof.} We start with the lower bound $\Delta_n^+(\p) \gtrsim \sup_{i \geq 1} \varepsilon_{i, p(i)}$. For any index $i\geq 1$ it holds that $\p \geq \step{i, p(i)}$ element-wise and thus we can prove that we have $\Delta_n^+(\p) \gtrsim \Delta_n^+\left(\step{i, p(i)}\right)$ since the tails at the individual coordinates are heavier for $\p$. Now pick index $J$ such that $\varepsilon:= \varepsilon_{J, p(J)} \geq \frac{1}{2} \sup_{i \geq 1} \varepsilon_{i, p(i)}$,  and let $q = p(J)$; then for independent $Y_i \sim \mathcal{B}(n,q)$:
 \begin{equation*}
\Pbb\paren{ \max_{i\leq J} \left\{\frac{Y_i}{n}\right\} \geq q + \varepsilon} = 1-\left( 1-\Pbb\paren{\frac{Y_i}{n} \geq q + \varepsilon)} \right)^J \geq  1- \left(1-\frac{c_0}{2J}\right)^J \geq 1-e^{-c_0/2}.
\end{equation*}
 We finish by applying Markov's inequality:
\begin{equation*}
\Ebb\sqb{ \max_{i\leq J}\abs{\frac{Y_i}{n}-q}}  \geq \Ebb\sqb{ \max_{i\leq J} \left\{\frac{Y_i}{n}\right\} - q} \geq 
\varepsilon \Pbb\paren{ \max_{i\leq J} \left\{\frac{Y_i}{n}\right\} \geq q + \varepsilon} \asymp \varepsilon.
\end{equation*}
That is, $\Delta_{n}(\p) \gtrsim \varepsilon_{J,q} \asymp \sup_{i \geq 1} \varepsilon_{i, p(i)}$.

We next turn to the upper bound $\Delta_n^+(\p) \lesssim \sup_{i \geq 1} \varepsilon_{i, p(i)}$. We use a pair of tight concentration and anti-concentration inequalities to estimate tail probabilities of binomial random variables and then we upper bound the expectation by tail-summation. Let $\varepsilon = \sup_{i\geq 1} \varepsilon_{i, p(i)}(n)$. Thus, for every index $i$ it holds that 
\begin{equation*}
\Pbb\paren{\hat{p}(i) \geq p(i) + \varepsilon} \lesssim \frac{1}{2i}.
\end{equation*}
On the other hand, the anti-concentration inequality from Lemma~\ref{lemma:anti_concentration} for $C \geq 1$ gives
\begin{equation*}
\Pbb\paren{\hat{p}(i) \geq p(i) + \varepsilon} \geq c_0 e^{-CnD(p(i)+\varepsilon \parallel p(i))}.
\end{equation*}
Combining both estimates results in
\begin{equation*}
D(p(i) + \varepsilon \parallel p(i)) \gtrsim \frac{\ln(2i)}{Cn}.
\end{equation*}
By a convexity argument on the KL-divergence we obtain $D(p(i)+kC\varepsilon \parallel p(i) ) \gtrsim k \ln(2i)/n$ for $k\geq 1$. Together with the standard Chernoff bound (Lemma \ref{lemma:chernoff_bound})
\begin{equation*}
    \Pbb\paren{\hat p(i)  \geq p(i) + kC\varepsilon} \leq e^{-nD(p(i)+kC\varepsilon \parallel p(i))} \le \frac{1}{(2i)^k}.
\end{equation*}
Then, by the union bound, for any $k\geq 2$,
\begin{equation*}
    \Pbb\paren{\sup_{i \geq 1} \left\{\hat p(i) - p(i) \right\} \geq kC\varepsilon}  \leq \sum_{i\geq 1} \frac{1}{(2i)^k} \leq \frac{1}{2^{k-1}}.
\end{equation*}
Finally, summing the tails yields the desired bound $\Delta_n^+(\p) =   \Ebb\sup_{i\geq 1} \left[\hat p(i) - p(i)\right]_+ \lesssim \varepsilon.$ Note that in the proof of this upper bound, we only needed the union bound. Hence, the upper bound also applies to general non-product distributions $\mu$ on $\{0,1\}^\Nbb$. \hfill $\blacksquare$

\subsection{Estimating the expected deviation for step-like sequences}

We next give estimates on the quantities  $\varepsilon_{J,q}$.
\begin{proposition}
\label{prop:combined_bound_eps}
    There exists a universal constants $C_1\geq 1$ and $c_2>0$ such that for all $n,J\geq 1$ and $q\in (0,1/2]$,
    \begin{equation*}
        \varepsilon_{J,q}(n) \leq C_1 \cdot \phi_{J,q}(n).
    \end{equation*}
    Further, if $1-(1-q)^n> \frac{c_0}{2J}$ (e.g. for $q\geq \frac{c_0}{nJ}$), one has
    \begin{equation*}
        \varepsilon_{J,q}(n) \geq \paren{\frac{1}{n}-q}\lor c_2\cdot  \phi_{J,q}(n).
    \end{equation*}
    On the other hand, if $1-(1-q)^n\leq\frac{c_0}{2J}$, we have $\varepsilon_{J,q}(n) = -q$.
\end{proposition}

\paragraph{Sketch of proof.} Informally, the concentration and anti-concentration inequalities from Lemma \ref{lemma:chernoff_bound} and Lemma \ref{lemma:anti_concentration} show that in most cases one has
\begin{equation*}
    \ln\paren{\frac{c_0}{2J}} \approx \ln \Pbb\paren{\frac{Y}{n} \geq q + \varepsilon_{J,q}(n)} \asymp -nD(q + \varepsilon_{J,q}(n) \parallel q).
\end{equation*}
These estimates are not tight in the ``Poissonian'' regime when $q \lesssim \frac{1}{nJ}$ which has to be treated separately. Otherwise, $\varepsilon_{J,q}(n)$ is essentially a solution to $D(q+\varepsilon \parallel q) \asymp \frac{\ln (J+1)}{n}$ in $\varepsilon$. The KL divergence shows two major regimes: either $D(q+\varepsilon \parallel q)\asymp \varepsilon\ln\frac{\varepsilon}{q}$ or $D(q+\varepsilon \parallel q)\asymp \frac{\varepsilon^2}{q}$. As a remark, these two regimes for the KL divergence are equivalent to the two standard regimes for Bennett's inequality (see \citet[Subsection 2.7]{boucheron2013concentration} for a more detailed overview of this inequality, or Lemma~\ref{lemma:estimates_kl} for a precise statement). These two asymptotic behaviors translate into the second and third regimes in the definition of $\phi_{J,q}(n)$ respectively.
\subparagraph{Case 1:}  $\frac{\ln(J+1)}{n}\asymp  D(q+\varepsilon_{J,q}(n) \parallel q) \asymp \varepsilon_{J,q}(n)\ln\frac{\varepsilon_{J,q}(n)}{q}$. In this case we obtain $\varepsilon_{J,q}(n) \asymp \frac{\ln(J+1)}{n \ln\left(\frac{\ln(J+1)}{nq}\right)}$.

\subparagraph{Case 2:} $\frac{\ln(J+1)}{n} \asymp D(q+\varepsilon_{J,q}(n) \parallel q) \leq \frac{\varepsilon_{J,q}(n)^2}{q}.$ This corresponds to a sub-Gaussian regime and we obtain $\varepsilon_{J,q}(n) \asymp \sqrt\frac{q \ln(J+1)}{n}.$ \hfill $\blacksquare$
\vspace{2mm}

\noindent We are now ready to complete the proof of \cref{thm:main_result}. 

\paragraph{Sketch of proof of \cref{thm:main_result}}
When some entries of $p$ are sufficiently large so that one can use the estimates from Proposition \ref{prop:combined_bound_eps} we combine it with Proposition \ref{prop:reduction}. This shows that $\Delta_n^+(\p) \asymp \sup_{i\geq 1}\phi_{i,p(i)}(n)$ whenever there exists $j\geq 1$ with $p(j)\geq \frac{c_0}{nj}$.

We treat separately the remaining case when for all $j\geq 1$, one has $p(j)\leq \frac{c_0}{nj}$. This corresponds to a Poissonian regime and it suffices to characterize the probability that one of the coordinates $\hat p_n(i)$ for $i\geq 1$ is non-zero. In this case, we obtain $\Delta_n^+(\p) \asymp \frac{1}{n}\land \sum_{j\geq 1}p(j).$ Proving that the leading term in $\Delta_n(\p) \asymp \Delta_n^-(\p) + \Delta_n^+(\p)$ is indeed $\Delta_n^+(\p)$ ends the proof of our main characterization in \cref{thm:main_result}. \hfill $\blacksquare$

\subsection{Discussion and comparison with bounds from the literature.}\label{ssec:comparison}

We first give some intuition on the decay of $\Delta_n(\p)$ given in \cref{thm:main_result}. The first case when for all $j\geq 1$, one has $p(j)\leq 1/(2nj)$ corresponds to rare events scenarios such that with high probability, $\hat{p}(j)\leq 1/n$ for all $j\geq 1$. This is characterized by the term $1/n$ from the bound $\Delta_n(\p)\asymp 1/n\land \sum_{j\geq 1}p(j)$. The second term characterizes the probability of the event when $\sup_{j\geq 1}\hat p_n(j) \geq \frac{1}{n}$. In this low-probability regime, the probabilities of success can be summed: with probability $\asymp 1\land \sum_{j\geq 1}np(j)$, at least one of the binomials $n\hat p_n(j)$ is nonzero.

We next turn to the second case when there exists $j\geq 1$ for which $p(j)\geq 1/(2nj)$. In this case, it is useful to compare our bounds using the functionals $S(\p)$ and $T(\p)$ from the literature. In particular, as a direct consequence of the characterization, we can recover the lower bound
\begin{equation*}
    \Delta_n(\p)\gtrsim 1\land \frac{T(\p)}{n}
\end{equation*}
from \cite{cohen2022local}. By definition of $\phi_{j,p(j)}$, for $n\leq T(\p)=\sup_{j\geq 1}\frac{\ln(j+1)}{\ln1/p(j)}$, we have $\Delta_n(\p)\asymp \sup_{j\geq 1}\phi_{j,p(j)}(n) = 1$. When $n>T(\p)$, the functions $\phi_{j,p(j)}(n)$ fall in either of the last two regimes (see Eq~\eqref{eq:definition_phi}), hence
\begin{align*}
    \Delta_n(\p) &\asymp \sqrt{\frac{S(\p)}{n}} \lor \frac{1}{n}\sup_{j\geq 1}\frac{\ln(j+1)}{\ln\paren{2+\frac{\ln(j+1)}{np(j)}}}\\
    &\gtrsim \sqrt{\frac{S(\p)}{n}} \lor \frac{1}{n}\sup_{j\geq 1}\frac{\ln(j+1)}{\ln\frac{\ln 1/p(j)}{p(j)}} \asymp \sqrt{\frac{S(\p)}{n}} \lor \frac{T(\p)}{n}.
\end{align*}
Together with the previous case, this shows that
\begin{equation*}
    \Delta_n(\p) \gtrsim 1\land \paren{\sqrt{\frac{S(\p)}{n}} \lor \frac{T(\p)}{n}}.
\end{equation*}
As suggested by the derivation, this lower bound is tight for $n$ in the neighborhood of $T(\p)$. For instance consider a step-like parameter $\step{J,q}$ with $q\geq \frac{1}{2nJ}$. Fix a constant $0\leq a<1$, then for any $T(\p)=\frac{\ln(J+1)}{\ln 1/q}\leq n\leq \frac{\ln(J+1)}{q^a}$, \cref{thm:main_result} implies
\begin{equation*}
    \frac{T(\p)}{n} \lesssim  \phi_{J,q}(n)\asymp \Delta_n(\p) \lesssim \frac{\ln(J+1)}{n\ln 1/q^{1-a}} =\frac{1}{1-a}\frac{T(\p)}{n}.
\end{equation*}

In terms of upper bounds, we recover the bound Eq~\eqref{eq:bound_cohen}. To do so, we give an upper bound of the functions $\phi_{J,q}(n)$ for $q\geq \frac{1}{2nJ}$ in the second regime from Eq~\eqref{eq:definition_phi} for which $\frac{\ln(J+1)}{\ln 1/q}\leq n\leq \frac{\ln(J+1)}{eq}$. First, note that the previous equation shows that for $n\leq \frac{\ln(J+1)}{\sqrt q}$, one has $\Delta_n(\p)\lesssim \frac{T(\p)}{n}$. Therefore it remains to consider the case when $\frac{\ln(J+1)}{\sqrt q}\leq n\leq \frac{\ln(J+1)}{eq}$. Note that $\ln\frac{1}{q} \leq 2\ln \frac{n}{\ln(J+1)}\leq 2\ln n$. Hence, in that regime, 
\begin{equation*}
    \phi_{J,q}(n)  \leq \frac{\ln(J+1)}{n} \lesssim \frac{\ln(J+1)}{n\ln\frac{1}{q}} \ln n.
\end{equation*}
Together with \cref{thm:main_result}, this implies
\begin{equation*}
    \Delta_n(\p) \asymp \sup_{j\geq 1}\phi_{j,p(j)}(n) \asymp \sup_{j\geq 1, p(j)\geq \frac{1}{2nj}}\phi_{j,p(j)}(n) \lesssim 1\land \paren{\sqrt{\frac{S(\p)}{n}} + \frac{T(\p)}{n}\ln n}.
\end{equation*}
Here, we used the fact that terms $\phi_{j,p(j)}(n)$ for which $p(j)\leq \frac{1}{2nj}$ are not dominant. This is formally shown in \cref{sec:proof_characterization} in the proof of Proposition \ref{prop:large_p}. Note that the estimates are tight for $n$ in the neighborhood of the beginning of the sub-Gaussian regime when the term $\sqrt{\frac{S(\p)}{n}}$ dominates.

As a summary of this discussion, assuming that there exists $j\geq 1$ for which $p(j)\geq \frac{1}{2nj}$, the decay of $\Delta_n(\p)$ shows three main regimes:
\begin{itemize}
    \item $\Delta_n(\p)\asymp 1$ when $n\leq T(\p)$,
    \item a somewhat sub-exponential regime when the decay of $\Delta_n(\p)$ interpolates between $\frac{T(\p)}{n}$ towards the start, and $\frac{T(\p)\ln n}{n}$ towards the end of this regime,
    \item the asymptotic sub-Gaussian regime $\Delta_n(\p)\asymp \sqrt{\frac{S(\p)}{n}}$.
\end{itemize}

\section{Expected maximum deviation for arbitrarily correlated distributions}
\label{sec:correlated_case}

In this section, we prove our estimate on the expected maximum empirical deviation for correlated distributions $\mu$ on $\{0,1\}^\Nbb$ from Corollary \ref{cor:correlated_case}. The latter requires the localized version of the classical Dvoretzky-Kiefer-Wolfowitz theorem given in \cref{thm:local_dkw}. For our purposes, we only need the result on intervals $(-\infty,x_0]$.

\begin{corollary}\label{cor:local_dkw}
    Let $X_1,\ldots,X_n$ be i.i.d.\ samples and denote by $F$ (respectively $F_n$) the true CDF (respectively empirical CDF). Then, for any $x_0\in\Rbb$ and $t\geq 0$,
    \begin{equation*}
        \Pbb\paren{\sup_{x \leq x_0} |F_n(x) - F(x)| > t \sqrt{\frac{F(x_0)}{n}} } \leq c_1 e^{-c_2 \min (t^2,t\sqrt{n F(x_0)})},
    \end{equation*}
    for some universal constants $c_1,c_2>0$.
\end{corollary}
The proof of both \cref{thm:local_dkw} and Corollary~\ref{cor:local_dkw} are given in \cref{sec:proof_local_dkw}.

We start by proving Corollary \ref{cor:correlated_case} that gives estimates of $\Delta_n(\mu)$ for general distributions $\mu$ on $\{0,1\}^\Nbb$.

\vspace{2mm}

\begin{proof}{\textbf{of Corollary \ref{cor:correlated_case}}}
All the upper bounds derived in the proof of \cref{thm:main_result} either used the union bound or Markov's inequality. We point in particular to \cref{prop:reduction} which gives the main upper bound whenever there exists $i\geq 0$ for which $\varepsilon_{i,p(i)}(n)\geq 0$. As a result, these still hold in the case of general distributions $\mu$ on $\{0,1\}^d$ with mean $p=\Ebb_{X\sim \mu}[X]$. We also provide some simple lower bounds which correspond from only considering the deviation from the first coordinate.
\begin{equation*}
    \Delta_n(\mu) = \Ebb \|\hat \p_n-\p\|_\infty \geq \Ebb |\hat p_n(1)-p(1)| \asymp \frac{np(1)\land \sqrt{np(1)}}{n}=p(1)\land\sqrt{\frac{p(1)}{n}}
\end{equation*}
The last estimate is classical and can be found for instance in \cite{berend2013sharp}.
\end{proof}

While the bounds provided in Corollary \ref{cor:correlated_case} may not be tight, the following result shows that if one only has access to the mean statistic $p=\Ebb_{X\sim\mu}[X]$, these are tight.

\begin{proposition}
    Let $p\in[0,\frac{1}{2}]^\Nbb_{\downarrow 0}$. There exists a distribution $\mu$ on $\{0,1\}^\Nbb$ with $\Ebb_{X\sim\mu}[X]=p$ such that for all $n\geq 1$,
    \begin{equation*}
        \Delta_n(\mu)\asymp p(1) \land \sqrt{\frac{p(1)}{n}}.
    \end{equation*}
\end{proposition}

\begin{proof}
    We start by constructing the corresponding distribution for $X\sim\mu$. We use a standard coupling which allows having samples $X$ with non-increasing coordinates. Precisely, let $U\sim\Ucal([0,1])$ be a uniform random variable. We let $X(j) = \1[U\leq p(j)]$. We now show that this distribution satisfies $\Delta_n(\mu)\asymp p(1)\land \sqrt{\frac{p(1)}{n}}$. With $n$ i.i.d.\ samples $X_i\sim\mu$, we define $Y=\sum_{i=1}^n X_i$ and $\hat p_n=\frac{Y}{n}$. Because of discretization issues, we distinguish several cases.

    We start by proving the upper bound $\Delta_n(\mu)\lesssim p(1)$. Suppose that $\hat p_n(1)\geq 2p(1)$. Then, we have for any $j\geq 2$, $|\hat p_n(j) -p(j)| \leq p(1) + |\hat p_n(1)-p(1)| \leq 2|\hat p_n(1)-p(1)| $. Hence, 
    \begin{equation*}
        \Delta_n(\p)\leq 2p(1) + 2\Ebb|\hat p_n(1)-p(1)|.
    \end{equation*}
    Because $\Ebb|\hat p_n(1)-p(1)|\asymp p(1)\land \sqrt{\frac{p(1)}{n}}$, we obtain the desired bound $\Delta_n(\mu)\lesssim p(1)$.

    We now turn to the upper bound $\Delta_n(\mu)\lesssim \sqrt{\frac{p(1)}{n}}$. Without loss of generality, we can therefore suppose that $p(1) \geq \frac{1}{n}$. Note that if $F_n(\cdot)$ is the empirical cumulative distribution function obtained from the i.i.d.\ uniform samples $U_1,\ldots,U_n$ used to define the variables $X_1,\ldots,X_n$,
    \begin{equation*}
        \Delta_n(\mu) = \Ebb \sup_{j\geq 1} |F_n(p(j))-p(j)| \leq \Ebb \sup_{u\leq p(1)} |F_n(u)-u|.
    \end{equation*}
    We then obtain bounds on the right-hand side using the localized version of the DKW theorem from Corollary \ref{cor:local_dkw}. Recalling that $p(1)\geq \frac{1}{n}$, this result implies that for $t>0$,
    \begin{equation*}
        \Pbb\paren{\sup_{u\in[0,1]}|F_n(u)-u| >t\sqrt{\frac{p(1)}{n}}} \leq c_1 e^{-c_2 t}
    \end{equation*}
    for universal constants $c_1,c_2>0$. As a result, we obtain
    \begin{equation*}
         \Delta_n(\mu) \leq \Ebb \sup_{u\in[0,p(1)]}|F_n(u)-u| \lesssim \sqrt{\frac{p(1)}{n}}.
    \end{equation*}
    This ends the proof of the proposition.
\end{proof}

\section{On the open problem from \citet{cohen2022local}}
\label{sec:open_problem}

As a consequence of our characterization, we can answer the COLT open problem posed by \cite{cohen2023open}, showing that the $\ln n$ factor the bound Eq~\eqref{eq:bound_cohen} is necessary.

\thmopenquestion*

\begin{proof}
    Consider the sequence 
    \[
    \begin{cases}
        p(j) = \frac{1}{K\sqrt n}:=p, \quad \text{if } \ln(j+1) \leq K\sqrt n\\
        p(j) = 0, \quad\text{otherwise}
    \end{cases}
    \]
    for a parameter $2\leq K\leq \sqrt n$ to define later. We denote by $J$ be the largest integer such that $p(J)=p$, in particular, we still have $\ln(J+1) \leq K\sqrt n$. Then,
    \begin{equation*}
        p(J) =\frac{1}{K\sqrt n}\geq \frac{\ln(J+1)}{K^2n} \geq \frac{1}{2K^2n}.
    \end{equation*}
    Now note that $\ln(J+2)>K\sqrt n\geq \sqrt n$, hence $K\leq \sqrt{J}$ for any $n\geq n_1$ for some constant $n_1\geq 1$. Together with the previous equation, we have $p(J)\geq \frac{1}{2Jn}$ so from \cref{thm:main_result} it follows
    \begin{align*}
    \Delta_n(\p) &\asymp 1\land \sup_{j \geq 1} \left(\frac{\ln(j+1)}{n\ln\left(2+\frac{\ln(j+1)}{np}\right)}  + \sqrt{\frac{p(j)\ln(j+1)}{n}} \right)\\
    &\asymp  \frac{K}{\sqrt n\ln K}  + \frac{1}{\sqrt n}.
    \end{align*}
    Hence, for $K\gtrsim C^2$ which can be achieved for $n$ sufficiently large
    (depending on $C$),we obtain $\frac{1}{2}\Delta_n(\p) \geq C\sqrt{\frac{S(\p)}{n}}$. On the other hand, note that
    \begin{equation*}
        T(\p) = \frac{\ln(J+1)}{\ln\frac{1}{p}} \leq \frac{2K\sqrt n}{\ln n}.
    \end{equation*}
    Therefore, we obtain $\psi(n)\gtrsim \frac{\ln n}{\ln K} \asymp \frac{\ln n}{\ln C}$.
\end{proof}

Although the previous proof used the general characterization of $\Delta_n(\p)$, \cref{thm:open_problem} can be proved with elementary arguments. We provide below a simple proof to obtain the same lower bound on $\Delta_n(\p)$.\\

\begin{proof}{\textbf{(elementary) of \cref{thm:open_problem}}}
    We use the same example for which $p(j)=\frac{1}{K\sqrt n}:=p$ if $\ln(j+1)\leq K\sqrt n$, and $p(j)=0$ otherwise. The parameter $4\leq K\leq \sqrt n$ will be fixed later. We denote by $J$ the largest integer such that $p(J)=p$. Let $l=\floor{\frac{K\sqrt n}{4\ln K}}$. For any $j\leq J$,
    \begin{align*}
        \Pbb\paren{\hat p_n(j) \geq \frac{l}{n}} &\geq \binom{n}{l}p^{l}(1-p)^{n-l}\\
        &\geq \paren{\frac{n}{l}}^l p^l (1-p)^n\\
        &\geq \exp\paren{-l\ln\frac{l}{np} -2np}\\
        &\geq \exp\paren{-\frac{K\sqrt n}{4\ln K}\ln\frac{K^2}{4\ln K} -\frac{2\sqrt n}{K}}\\
        &\geq \exp\paren{-\frac{1}{2}K\sqrt n -\frac{2\sqrt n}{K}}\geq e^{-\frac{3}{4}K\sqrt n}.
    \end{align*}
    In the third inequality, we used the fact that $\ln(1-x)\geq -2x$ for $x\in[0,\frac{1}{2}]$. Now note that $\ln(J+2)>K\sqrt n$. Therefore, 
\begin{align*}
    \Pbb\sqb{\max_{j\in[J]} \hat p_n(j) \geq \frac{l}{n}} &\geq 1-\paren{1-\Pbb\sqb{\hat p_n(1) \geq \frac{l}{n}}}^J\\
    &\geq 1-\exp\paren{-J\cdot \Pbb\sqb{\hat p_n(1) \geq \frac{l}{n}}}\\
    &\geq 1-\exp\paren{-(e^{K\sqrt n}-2)e^{-\frac{3}{4}K\sqrt n}} \geq 1-e^{2/e^2}\geq \frac{1}{2}.
\end{align*}
Hence, for $n$ sufficiently large
\begin{equation*}
    \Delta_n(\p) \geq \frac{1}{2}\paren{\frac{l}{n}-p} \geq \frac{K}{10\sqrt n \ln K}.
\end{equation*}
This provides the same lower bound for $\Delta_n(\p)$ as in the previous proof up to constants, and the proof is identical from that point.
\end{proof}

\section{Conclusion and future work}
\label{sec:conclusion}

In this paper, we have derived the exact characterization (up to a constant factor) of the infinite-norm deviation of the empirical mean of the distribution supported on $\{0,1\}^d$ from the true mean in the case of product distributions. For the case of general (non-product) distributions, we have derived a lower and upper bound on the deviation when we only have access to the mean statistics, and provided distributions corresponding to these bounds. Along the way, we proved a localized version of Dvoretzky–Kiefer–Wolfowitz inequality. We extended the results to the cases where the deviation is measured in a general $q-$norm and provided characterization of the convergence, and both finite and assymptotic bounds on the convergence. Additionally, we considered the case where the random variables were supported on $[0,1]^d$ instad of $\{0,1\}^d$ and we have derived a lower and upper bound on the deviation when we only have access to the mean (or variance) statistics, and provided distributions corresponding to these bounds.

An interesting direction for future work would be to consider the case of dependent coordinates with information about how they are dependent; e.g., when we know the covariance matrix. Exact non-asymptotic bounds would be of particular interest, but to the best of our knowledge even characterizing the asymptotic behavior of the maximum deviation for arbitrary distributions is an open question.

\acks{
The authors would like to mention that concerning the COLT 2023 open problem \cite{cohen2023open} of removing the factor $\ln n$ in Eq~\eqref{eq:bound_cohen}, three independent examples were first (almost) simultaneously given by Jaouad Mourtada and both authors of this paper, showing that a factor $\ln n/\ln\ln n$ is at least necessary in general. These examples were then improved in \cref{thm:open_problem} of this paper to show that the full $\ln n$ factor is necessary. The authors would like to express their sincere gratitude to Aryeh Kontorovich for carefully updating a timeline of that progress at \url{https://mathoverflow.net/questions/447472} and for invaluable discussions on Local Glivenko-Cantelli bounds; and to Jaouad Mourtada for useful discussions and advice on the exposition of these examples. VV was supported by the DFG Cluster of Excellence ``Machine Learning – New Perspectives for Science”, EXC 2064/1, project number 390727645 and is thankful for the support of Open Philanthropy.}

\bibliography{refs}

\appendix


\section{Proofs of Section \ref{sec:characterization}}
\label{sec:proof_characterization}

We start by giving a proof of Proposition \ref{prop:reduction} that allows for reducing the problem of characterizing the expected maximum empirical mean deviation for general probability vectors $p\in[0,\frac{1}{2}]^\Nbb_{\downarrow 0}$ to step-like vectors $\step{J,q}$ constant equal to $q$ until coordinate $J$ then zero afterwards.

\vspace{2mm}

\begin{proof}{\textbf{of Proposition \ref{prop:reduction}}}
    We start with the lower bound. 
    Here, we will mainly show that  Bernoulli random variables with higher mean (which is smaller than $0.5$) cannot have much lighter tails compared to the lower mean ones; thus, when we would have that $\Delta_n(\p) \geq \Delta_n(\step{i, p(i)})$ as $\p \geq \step{i, p(i)}$ coordinate-wise.

    Fix $i\geq 1$ such that $\varepsilon_i = \varepsilon_{i,p(i)}(n) \geq \frac{1}{2}\sup_{j\geq 1} \varepsilon_{j,p(j)}(n)$. Let $j\leq i$. By construction, one has $p(j)\geq p(i)$. If $p(j) \leq \frac{\varepsilon_i}{2}$, one has
    \begin{equation}\label{eq:eq1}
        \Pbb\paren{\hat p_n(j) \geq p(j)+\frac{\varepsilon_i}{2}} \geq \Pbb(\hat p_n(j) \geq \varepsilon_i) \geq \Pbb(\hat p_n(i) \geq p(i) + \varepsilon_i) > \frac{c_0}{2i}.
    \end{equation}
    We now suppose that $p(j)>\frac{\varepsilon_i}{2}$. Next, because $p(i) + \varepsilon_i \in \frac{1}{n}\Zbb$ and $\varepsilon_i>0$, we obtain $p(j)+\varepsilon_i\geq p(i)+\varepsilon_i\geq\frac{1}{n}$. 
    
    We first treat the case when $p(j)>\frac{1}{n}$. First, the Chernoff bound from Lemma \ref{lemma:chernoff_bound} shows that
    \begin{equation*}
        \frac{c_0}{2i} < \Pbb(\hat p_n(i) \geq p(i) + \varepsilon_i) \leq e^{-nD(p(i) + \varepsilon_i\parallel p(i))}.
    \end{equation*}
    Hence, $D(p(i) + \varepsilon_i\parallel p(i)) \leq \frac{1}{n}\ln\frac{2i}{c_0}$. Now note that the function $p\in[0,\frac{1-\varepsilon_i}{2}]\mapsto D(p+\varepsilon_i\parallel p)$ is non-increasing (and convex), so that if $p(j) \leq \frac{1-\varepsilon_i}{2}$, one has
    \begin{equation*}
        D(p(j) + \varepsilon_i\parallel p(j)) \leq D(p(i) + \varepsilon_i\parallel p(i)).
    \end{equation*}
    On the other hand, if $p(j)>\frac{1-\varepsilon_i}{2}$, then
    \begin{equation*}
        D\paren{p(j) + \frac{\varepsilon_i}{2}\parallel p(j)} \leq D\paren{\frac{1+\varepsilon_i}{2}\parallel p(j)} \leq D\paren{\frac{1+\varepsilon_i}{2}\parallel \frac{1-\varepsilon_i}{2}} \leq D(p(i) + \varepsilon_i\parallel p(i)),
    \end{equation*}
    where in the last inequality, we also used the fact that $p\in[0,\frac{1-\varepsilon_i}{2}]\mapsto D(p+\varepsilon_i\parallel p)$ is non-increasing. In both cases, using the convexity of the KL-divergence in the first argument, we obtain
    \begin{equation*}
        D\paren{p(j) + \frac{\varepsilon_i}{2C}\parallel p(j)}  \leq \frac{1}{C} D\paren{p(j) + \frac{\varepsilon_i}{2}\parallel p(j)}\leq \frac{1}{nC}\ln\frac{2i}{c_0}.
    \end{equation*}
    Now because $p(j)\geq \frac{1}{n}$, we can use Lemma \ref{lemma:anti_concentration} (without loss of generality we can suppose $C\geq 2$ so that $p(j) + \frac{\varepsilon_i}{2C} \leq p(j)+\frac{1}{4} \leq \frac{1+p(j)}{2}$ which ensures that we can apply Lemma \ref{lemma:anti_concentration}) which gives
    \begin{equation}\label{eq:eq2}
        \Pbb\paren{\hat p_n(j) \geq p(j) + \frac{\varepsilon_i}{2C}} \geq \frac{c_0^2}{2i}.
    \end{equation}

    It remains to consider the case when $p(j)\leq \frac{1}{n}$. Recall that we have $p(j)+\varepsilon_i\geq \frac{1}{n}$ and $p(j)\geq \frac{\varepsilon_i}{2}$. As a result, $\frac{1}{3n}\leq p(j)\leq \frac{1}{n}$, and $\varepsilon_i\leq \frac{2}{n}$. Now observe that for $n\geq 2$,
    \begin{align*}
        \Pbb\paren{\hat p_n(j)\geq p(j) + \frac{1}{n}} &\geq \Pbb\paren{\hat p_n(j)\geq \frac{2}{n}} \\
        &\geq \binom{n}{2}p(j)^2 (1-p(j))^{n-2}\\
        & \geq \frac{n^2}{4} \frac{1}{(3n)^2}\paren{1-\frac{1}{n}}^n \geq \frac{1}{144}.
    \end{align*}
    In particular, this shows that $\Delta_n^+(\p) \geq \frac{1}{144n} \geq \frac{\varepsilon_i}{288}$. This shows that the lower bound is directly achieved whenever there exists such an index $j$. Otherwise, the previous cases in Eq~\eqref{eq:eq1} and Eq~\eqref{eq:eq2} showed that for all $j\leq i$,
    \begin{equation*}
        \Pbb\paren{\hat p_n(j) \geq p(j) + \frac{\varepsilon_i}{2C}} \geq \frac{c_0^2}{2i}.
    \end{equation*}
    Then,
    \begin{align*}
        \Pbb\paren{\sup_j [\hat p_n(j) - p(j)] \geq \frac{\varepsilon_i}{2C}} \geq 1-\paren{1-\frac{c_0^2}{2i}}^i \geq 1-e^{-c_0^2/2}>0.
    \end{align*}
    In particular, this shows that
    \begin{equation*}
        \Delta_n^+(\p) \geq \frac{1-e^{-c_0^2/2}}{2C}\varepsilon_i.
    \end{equation*}
    This gives the desired lower bound $\Delta_n^+(\p) \geq c\cdot \sup_{i\geq 1}\varepsilon_{i,p(i)}(n),$ for some universal constant $c>0$.

    We now turn to the upper bound. 
    Here, we show that the probability that the deviation at position $i$ exceeds $C\varepsilon$ by a factor of $k$ is at most $(2i)^{-k}$. Thus, decaying very quickly in both $k$ and $i$. We union bound this probability over the coordinates and sum up the tails (over $k$) to show $\Delta_n^+(\p) \lesssim \varepsilon$.
    
    For convenience, define $\varepsilon = \sup_{j\geq 1}\varepsilon_{j,p(j)}(n)$, and let $\tilde \varepsilon = (\varepsilon\land \frac{1}{4})\lor\frac{1}{n} \geq \frac{\varepsilon}{4}$. As a result, for any $p\in(0,\frac{1}{2}]$, one has $p+\tilde\varepsilon \leq \frac{1+p}{2}$. Fix $i\geq 1$. We can then apply Lemma \ref{lemma:anti_concentration} since $\tilde\varepsilon\geq \frac{1}{n}$, and use the continuity of the KL-divergence to obtain
    \begin{equation*}
        \Pbb(\hat p_n(i) - p(i) > \tilde \varepsilon) \geq c_0 e^{-Cn D(p(i) + \tilde \varepsilon \parallel p(i))}.
    \end{equation*}
    On the other hand,
    \begin{equation*}
        \Pbb(\hat p_n(i) - p(i) > \tilde \varepsilon) \leq \Pbb(\hat p_n(i) - p(i) > \varepsilon_i) \leq \frac{c_0}{2i}.
    \end{equation*}
    Combining the two equations gives
    \begin{equation*}
        D(p(i) + \tilde \varepsilon \parallel p(i)) \geq \frac{\ln (2i)}{nC}.
    \end{equation*}
    Because the KL-divergence is convex in the first argument, for any $k\geq 1$ (with $C\geq1$), we have
    \begin{equation*}
        D(p(i) + kC\tilde\varepsilon \parallel p(i)) \geq kCD(p(i) + \tilde\varepsilon \parallel p(i)) \geq \frac{k\ln(2i)}{n}.
    \end{equation*}
    Now using the Chernoff bound from Lemma \ref{lemma:chernoff_bound},
    \begin{align*}
        \Pbb\paren{\hat p_n(i)\geq p(i) + kC\tilde\varepsilon} \leq e^{-nD(p(i) + kC\tilde\varepsilon\parallel p(i))} \leq \frac{1}{(2i)^k}.
    \end{align*}
    Using the union-bound yields for any $k\geq 2$,
    \begin{equation}\label{eq:upper_tail_prob}
        \Pbb\paren{\sup_{i\geq 1}[\hat p_n(i) - p(i)] \geq kC\tilde\varepsilon } \leq \frac{1}{2^k}\sum_{i\geq 1}\frac{1}{i^k} \leq \frac{1}{2^{k-1}}.
    \end{equation}
    In particular, we obtain
    \begin{equation*}
        \Delta_n^+(\p) = \Ebb\sqb{\sup_{i\geq 1}[\hat p_n(i) - p(i)] } \leq  2C\tilde\varepsilon + C\tilde\varepsilon\sum_{k\geq 2} \Pbb\paren{\sup_{i\geq 1}[\hat p_n(i) - p(i)] \geq kC\tilde\varepsilon } \leq 3C\tilde\varepsilon.
    \end{equation*}
    This already gives the desired upper bound whenever say $\varepsilon\geq \frac{1}{4n}$, since this implies $\tilde\varepsilon\leq 4\varepsilon$. We now consider the case when $\varepsilon <\frac{1}{4n}.$ As discussed above, there exists $i\geq 1$ such that $p(i)+\varepsilon\geq p(i) + \varepsilon_i \geq \frac{1}{n}$. In particular, $p(i) \geq \frac{3}{4n}$. Now note that for $n\geq 2$,
    \begin{equation*}
        \Pbb\paren{\hat p_n(i) \geq p(i)+ \frac{1}{n}} = \Pbb\paren{\hat p_n(i) \geq \frac{2}{n}} = 1-\paren{1-\frac{1}{n}}^n - \paren{1-\frac{1}{n}}^{n-1}\geq 1-\frac{1}{e}-\frac{1}{2}> \frac{c_0}{2i}.
    \end{equation*}
    As a result, we should have $\varepsilon_i\geq \frac{1}{n}$, which is contradictory. This shows that the upper bound holds in all considered cases, which ends the proof of the claim for product distributions.

    The upper bound directly holds for general distributions on $\{0,1\}^\Nbb$ because it only used the union bound to analyze the effect between coordinates.
\end{proof}

The next step is to characterize the quantities $\varepsilon_{J,q}(n)$. Before doing so, we state some simple bounds on the KL-divergence.

\begin{lemma}\label{lemma:estimates_kl}
    Let $0\leq q,\varepsilon\leq \frac{1}{4}$ and suppose $\varepsilon\geq 8q$. Recall that $h(u) = (1+u)\ln(1+u)-u$. Then,
    \begin{equation*}
         \frac{\varepsilon}{2} \ln\frac{\varepsilon}{q} \leq q\h\left(\frac\varepsilon q\right) \leq D(q+\varepsilon\parallel q)\leq 2\varepsilon \ln\frac{\varepsilon}{q}.
    \end{equation*}
    Also, for any $0\leq q,\varepsilon\leq 1$ with $q+\varepsilon\leq \frac{1}{2}$,
    \begin{equation*}
         \frac{\varepsilon^2}{2(q+\varepsilon)} \leq q\h\left(\frac\varepsilon q\right) \leq D(q+\varepsilon\parallel q)\leq \frac{\varepsilon^2}{q}.
    \end{equation*}
\end{lemma}

\begin{proof}
First we show that $q\h\left(\frac\varepsilon q\right) \leq D(q+\varepsilon \parallel q)$:
\begin{align*}
D(q+\varepsilon\parallel q) &= (q+\varepsilon)\ln\left(\frac{q+\varepsilon}{q}\right) + (1-q-\varepsilon)\ln\left(\frac{1-q-\varepsilon}{1-q}\right)\\
&= q\h\left( \frac{\varepsilon}{q}\right)+ \varepsilon + (1-q-\varepsilon)\ln\left(\frac{1-q-\varepsilon}{1-q}\right) \\&\geq q\h\left( \frac{\varepsilon}{q}\right),
\end{align*}
because  $\varepsilon + (1-q-\varepsilon)\ln\left(\frac{1-q-\varepsilon}{1-q}\right) = 0$ for $\varepsilon = 0$ and is increasing in $\varepsilon$ since $\frac{\partial}{\partial \varepsilon} \left(\varepsilon + (1-q-\varepsilon)\ln\left(\frac{1-q-\varepsilon}{1-q}\right)\right)= - \ln\left(\frac{1-q-\varepsilon}{1-q}\right) \geq 0$.

We have for $\varepsilon \geq 8q$
\begin{equation*}
    q\h\left(\frac \varepsilon q\right) \geq  \varepsilon\ln\frac{\varepsilon}{q} - \varepsilon 
    \geq \frac{\varepsilon}{2}\ln\frac{\varepsilon}{q} + \varepsilon\left(\frac{1}{2}\ln\frac \varepsilon q - 1\right) 
    \geq \frac{\varepsilon}{2}\ln\frac{\varepsilon}{q},
\end{equation*}

since $\ln \frac \varepsilon q \geq \ln 8 \geq 2$.
\comment{
    We have
    \begin{equation*}
        D(q+\varepsilon\parallel q) \geq (q+\varepsilon)\ln\frac{q+\varepsilon}{q} + \ln\paren{1-\frac{4\varepsilon}{3}} \geq \varepsilon\paren{\ln\frac{\varepsilon}{q}  - 2}.
    \end{equation*}
    where we used $\ln(1-x)\geq -3x/2$ for $x\leq 1/3$. Because $\varepsilon\geq 10q$, we therefore obtain
    \begin{equation*}
        D(q+\varepsilon\parallel q) \geq \paren{1-\frac{2}{\ln(10)}}\varepsilon\ln \frac{\varepsilon}{q} \geq \frac{\varepsilon}{8}\ln \frac{\varepsilon}{q}.
    \end{equation*}
    }
    On the other hand,
    \begin{equation*}
        D(q+\varepsilon\parallel q) \leq (q+\varepsilon)\ln\frac{q+\varepsilon}{q} \leq \frac{9}{8}\varepsilon \ln\frac{2\varepsilon}{q} \leq 2\varepsilon \ln \frac{\varepsilon}{q}.
    \end{equation*}
    
    We now turn to the second bound when $\varepsilon\leq q$. Letting $f(\varepsilon) = D(q+\varepsilon\parallel q)$ and $g(\varepsilon) = qh( \frac \varepsilon q )$), we have $f''(\varepsilon) = \frac{1}{q+\varepsilon} + \frac{1}{1-q-\varepsilon}$ and $g''(\varepsilon) = \frac{1}{q+\varepsilon}$). As a result, for any $x\in[0,\varepsilon]$,
    \begin{equation*}
        \frac{1}{q+\varepsilon} \leq g''(x) \leq f''(x) \leq \frac{2}{q+x}\leq \frac{2}{q},
    \end{equation*}
    An application of Taylor's expansion theorem ends the proof. 
\end{proof}

\begin{lemma}\label{lemma:compare_kl_h}
    Let $0\leq \varepsilon ,q\leq \frac{1}{2}$. Then,
    \begin{equation*}
         q\h\paren{\frac{\varepsilon}{q}}\geq D\paren{q+\frac{\varepsilon}{2}\parallel q}.
    \end{equation*}
\end{lemma}

\begin{proof}
    Let $f(\varepsilon) = q\h\paren{\frac{\varepsilon}{q}} - D\paren{q+\frac{\varepsilon}{2}\parallel q}$. Then, $f(0)=0$ and for any $0\leq \varepsilon\leq \frac{1}{2}$
    \begin{align*}
        f'(\varepsilon) &\geq \frac{1}{2} \ln \paren{\frac{q+\varepsilon}{q}} + \frac{1}{2}\ln \paren{\frac{1-q-\varepsilon/2}{1-q}}\\
        &\geq \frac{1}{2} \ln \paren{\frac{q(1-q)+\varepsilon(1-3q/2-\varepsilon/2)}{q(1-q)}} \geq 0.
    \end{align*}
    Hence, for any $0\leq \varepsilon\leq \frac{1}{2}$, we have $f(\varepsilon)\geq 0$.
\end{proof}

We now present bounds on $\varepsilon_{J,q}(n)$. To do so, we start by showing upper bounds using the function $\phi_{J,q}(n)$.

\begin{proposition}
\label{prop:upper_bound_eps}
    There exists a universal constant $C_1\geq 1$ such that for all $n,J\geq 1$ and $q\in (0,1/2]$,
    \begin{equation*}
        \varepsilon_{J,q}(n) \leq C_1 \cdot \phi_{J,q}(n).
    \end{equation*}
\end{proposition}

\begin{proof}
    The proof relies on the Chernoff bound from Lemma \ref{lemma:chernoff_bound}. In the rest of this proof, we let $Y\sim\Bcal(n,p)$, $\hat p_n = \frac{Y}{n}$ and $\varepsilon:=\varepsilon_{J,q}(n)$. We have $\frac{c_0}{2J} <\Pbb(\hat p_n \geq q + \varepsilon) \leq e^{-nD(q+\varepsilon\parallel q)}.$ As a result, this shows
    \begin{equation*}
        D(q+\varepsilon\parallel q) \leq \frac{\ln\frac{2J}{c_0}}{n}.
    \end{equation*}
    The upper bound given in the first regime $n\leq \frac{\ln(J+1)}{\ln\frac{1}{q}}$ is trivial. We then turn to the second regime. For convenience, we will denote $\phi(n):=\phi_{J,q}(n)$.
    
    \paragraph{Regime $\frac{\ln(J+1)}{\ln\frac{1}{q}}\leq n \leq  \frac{\ln(J+1)}{eq}$.} We first observe that the function $n\ln \frac{\ln(J+1)}{nq}$ is non-decreasing in that regime. As a result, we always have $\phi(n) \geq \phi\paren{\frac{\ln(J+1)}{eq}} = eq.$ Next, the upper bounds are immediate if $q\geq \frac{1}{4}$ since using a constant $C_1\geq 4$ would yield a trivial upper bound $1$. We therefore suppose that $q\leq \frac{1}{4}$. Similarly, without loss of generality, suppose $\phi(n)\leq \frac{1}{40}$. As a result, by Lemma \ref{lemma:estimates_kl}, for any constant $\alpha\geq 1$,
    \begin{equation*}
        D(q+3 \alpha \phi(n)\parallel q)\geq \frac{3}{2}\alpha\phi(n) \ln \frac{3\alpha\phi(n)}{q}.
    \end{equation*}
    Now if $x\geq 1$ is the solution to the equation $x\ln x = 2\frac{\ln\frac{2J}{c_0}}{nq}:=z\geq 2$, one has precisely
    \begin{equation*}
        x\asymp \frac{z}{\ln z} \asymp \frac{\ln(J+1)}{nq \ln\frac{\ln(J+1)}{nq}}.
    \end{equation*}
    As a result, there exists a constant $\alpha\geq 1$ sufficiently large such that either $3\alpha \phi(n) \geq 1/4$ (in which case the bound for this regime is immediate for sufficiently large $C_1$), or
    \begin{equation*}
        D(q+3 \alpha \phi(n)\parallel q) \geq \frac{\ln\frac{2J}{c_0}}{n}.
    \end{equation*}
    This implies $\varepsilon \leq 3\alpha \phi(n)$.

    \paragraph{Regime $ n \geq \frac{\ln(J+1)}{eq}$.}
    In this regime, we have $\phi(n) \leq q\sqrt{e}\leq 2q$. Using the second estimate from Lemma \ref{lemma:estimates_kl}, we have for any constant $\gamma\geq 1$,
    \begin{equation*}
        D(q+\gamma \phi(n)\parallel q) \geq \frac{\gamma^2\ln(J+1)}{2(1+2\gamma)n}.
    \end{equation*}
    As a result, there exists a universal constant $\gamma\geq 1$ such that $D(q+\gamma \phi(n)\parallel q) \geq 2\frac{\ln\frac{2J}{c_0}}{n}$, which implies $\varepsilon \leq \gamma\phi(n)$. This ends the proof of the proposition.
\end{proof}

We next turn to lower bounds.

\begin{proposition}
\label{prop:lower_bound}
    There is a universal constants $c_2>0$ such that for all $J,n\geq 1$ and $q\in (0,1/2]$ satisfying $1-(1-q)^n> \frac{c_0}{2J}$ (e.g. for $q\geq \frac{c_0}{nJ}$), one has
    \begin{equation*}
        \varepsilon_{J,q}(n) \geq \paren{\frac{1}{n}-q}\lor c_2\cdot  \phi_{J,q}(n).
    \end{equation*}
    On the other hand, if $1-(1-q)^n\leq\frac{c_0}{2J}$, we have $\varepsilon_{J,q}(n) = -q$.
\end{proposition}

\begin{proof}
    As in the previous proof, we let $\hat p_n =\frac{Y}{n}$ where $Y\sim\Bcal(n,p)$. We compute $\Pbb(\hat p_n \geq 1/n) = 1-(1-q)^n$. As a result, if $1-(1-q)^n> \frac{c_0}{2J}$, we have $q+\varepsilon_{J,q}(n)\geq \frac{1}{n}$ and otherwise, $q+\varepsilon_{J,q}(n)=0$. We now prove that $q\geq \frac{c_0}{nJ}$ suffices to obtain $1-(1-q)^n> \frac{c_0}{2J}$. Note that $1-(1-q)^n\geq 1-e^{-qn}$. If $q\geq \frac{\ln 2}{n}$, we have $\Pbb(\hat p_n \geq 1/n) \geq \frac{1}{2} > \frac{c_0}{2J}$. Otherwise, since $qn\leq \ln 2$ and the exponential function is convex, we have
    \begin{equation*}
        \Pbb(\hat p_n \geq 1/n) \geq \frac{nq}{2\ln 2} >\frac{c_0}{2J}.
    \end{equation*}
    
    We assume from now on that $1-(1-q)^n> \frac{c_0}{2J}$. Let $x(n)$ be the solution to the equation
    \begin{equation*}
        D(q+x(n)\parallel q) = \frac{\ln\frac{3J}{2}}{Cn}.
    \end{equation*}
    If $q+x(n)\geq \frac{1}{n}$, Lemma \ref{lemma:anti_concentration} shows that
    \begin{equation*}
        \Pbb(\hat p_n \geq q+x(n)) \geq c_0 e^{-CnD(q+x(n)\parallel q) } \geq \frac{2c_0}{3J} >\frac{c_0}{2J}.
    \end{equation*}
    As a result, if $q+x(n) \geq \frac{1}{n}$, we obtain $\varepsilon_{J,q}(n) \geq x(n)$. Thus, in both cases, we obtain
    \begin{equation*}
        \varepsilon_{J,q}(n) \geq \paren{\frac{1}{n}-q}\lor x(n).
    \end{equation*}
    It remains to compute an estimate of $x(n)$. Using Lemma \ref{lemma:estimates_kl}, if $x(n)\geq 8q$, we have $D(q+x(n)\parallel q)\asymp x(n)\ln\frac{x(n)}{q}$, so that similarly as in the proof of Proposition \ref{prop:upper_bound_eps}, we have with $z=\frac{\ln\frac{3J}{2}}{Cnq}$,
    \begin{equation*}
        x(n)\asymp q\frac{z}{\ln (2+z)}\asymp \frac{\ln(J+1)}{n\ln\paren{2+\frac{\ln(J+1)}{nq}}}.
    \end{equation*}
    On the other hand, if $x(n)\leq 10q$, the second bounds of Lemma \ref{lemma:estimates_kl} show that $D(q+x(n)\parallel q)\asymp \frac{x(n)^2}{q}$. As a result, this yields
    \begin{equation*}
        x(n)\asymp \sqrt{\frac{q\ln\frac{3J}{2}}{n}}\asymp \sqrt{\frac{q\ln(J+1)}{n}}.
    \end{equation*}
    The cutoff for $x(n)$ corresponds to $n\asymp\frac{\ln(J+1)}{eq}$, and the two estimates of $x(n)$ match in this complete regime (if $a\frac{\ln(J+1)}{eq} \leq n \leq b\frac{\ln(J+1)}{eq}$ for some universal constants $0<a\leq b$) up to constants. Recalling that $x(n)\leq 1$, we obtained exactly $x(n)\asymp \phi(n)$. This proves that for some universal constant $c_2>0$, one has
    \begin{equation*}
        \varepsilon_{J,q}(n) \geq \paren{\frac{1}{n}-q}\lor c_2\cdot \phi(n),
    \end{equation*}
    which ends the proof of the proof of the proposition.
\end{proof}

\begin{proof}{\textbf{of Proposition \ref{prop:combined_bound_eps}}}
  Propositions \ref{prop:upper_bound_eps} and \ref{prop:lower_bound} exactly prove Proposition \ref{prop:combined_bound_eps}.   
\end{proof}

We now combine the two results when possible, to give estimates on $\Delta_n^+(\p)$.

\begin{proposition}
\label{prop:large_p}
    For any $p\in[0,\frac{1}{2}]^\Nbb_{\downarrow 0}$ and $n\geq 1$ such that there exists $j\geq 1$ with $p(j)\geq \frac{c}{nj}$, we have
    \begin{equation*}
        \Delta_n^+(\p) \asymp \sup_{i\geq 1}\phi_{i,p(i)}(n).
    \end{equation*}
\end{proposition}

\begin{proof}
    For any $j\geq 1$ such that $p(j)\geq \frac{c_0}{nj}$, using Proposition \ref{prop:combined_bound_eps}, we have $\varepsilon_{j,p(j)}(n) \asymp \phi_{j,p(j)}(n)\geq 0$. In fact, whenever $\varepsilon_{j,p(j)}(n)\geq 0$, these propositions imply $\varepsilon_{j,p(j)}(n) \asymp \phi_{j,p(j)}(n)$. Then, Proposition \ref{prop:reduction} implies that $\Delta_n^+(\p)\asymp \sup_{i\geq 1}\varepsilon_{i,p(i)}(n)$. For convenience, let $\varepsilon = \sup_{i\geq 1}\varepsilon_{i,p(i)}(n)\geq \frac{c_3}{n}$. In order to prove the theorem, given Proposition \ref{prop:combined_bound_eps}, it remains to prove that if $\varepsilon_{i,p(i)}(n)<0$ for some $i\geq 0$, we have $\phi_{i,p(i)}(n)\lesssim \phi_{j,p(j)}(n)$ ($j\geq 1$ is such that $p(j)\geq \frac{c_0}{nj}$).

    First, necessarily $p(i)<\frac{c_0}{ni}$. As a result, $\frac{\ln(j+1)}{\ln\frac{1}{p(i)}}\leq 1$ and hence the first regime for $\phi_{i,p(i)}$ is not present. Further,
    \begin{equation*}
        \frac{\ln(i+1)}{ep(i)} \geq \frac{i\ln(i+1)}{ec_0} n \geq \frac{\ln 2}{ec_0}n.
    \end{equation*}
    This proves that either $n$ falls the second regime for $\phi_{i,p(i)}(n)$, i.e., $\frac{\ln(i+1)}{\ln\frac{1}{p(i)}}\leq n\leq \frac{\ln(i+1)}{ep(i)}$, or $n\asymp\frac{\ln(i+1)}{ep(i)}$. In both cases,
    \begin{equation*}
        \phi_{i,p(i)}(n) \asymp \frac{\ln(i+1)}{n\ln\paren{2+\frac{\ln(i+1)}{np(i)}}} \leq \frac{\ln(i+1)}{n\ln\paren{2+\frac{i\ln(i+1)}{c_0}}} \lesssim\frac{1}{n}.
    \end{equation*}
    As a result, there exists a universal constant $C_3>0$ such that $\phi_{i,p(i)}(n)\leq \frac{C_3}{n}$. Now recall that $p(j)\geq \frac{c_0}{nj}$. We first consider the case when $p(j)\leq \frac{\ln(j+1)}{en}$. In this case, $\phi_{j,p(j)}(n)$ lies in one of the two regimes. In the first regime, we have directly $\phi_{j,p(j)}(n)\asymp 1\gtrsim \frac{1}{n}$. In the second case, we have
    \begin{equation*}
        \phi_{j,p(j)}(n) = \frac{\ln(j+1)}{n\ln\frac{\ln(j+1)}{np(j)}} \geq \frac{\ln(j+1)}{n\ln\frac{j\ln(j+1)}{c_0}}\gtrsim \frac{1}{n}.
    \end{equation*}
    We now consider the case when $p(j)\geq \frac{\ln(j+1)}{en}$. In this case, $\phi_{j,p(j)}(n)$ lies in the third regime which yields
    \begin{equation*}
        \phi_{j,p(j)}(n) =\sqrt{\frac{p(j)\ln(j+1)}{n}}\geq \frac{\ln(j+1)}{\sqrt e \cdot n}\gtrsim \frac{1}{n}.
    \end{equation*}
    As a result, there is a constant $c_3>0$ such that in all cases $\phi_{j,p(j)}(n)\geq \frac{c_3}{n}$. Putting everything together yields
    \begin{equation*}
        \sup_{i\geq 1}\phi_{i,p(i)}(n) \asymp \sup_{j\geq 1, \varepsilon_{j,p(j)}(n)\geq 0} \phi_{i,p(i)}(n) \asymp \sup_{i\geq 1} \varepsilon_{j,p(j)}(n) \asymp \Delta_n^+(\p).
    \end{equation*}
    This ends the proof of the proposition.
\end{proof}

It remains to consider the Poissonian case in which one has $p(j)\leq \frac{c}{nj}$ for all $i\geq 1$. Recall that we have $c_0<\frac{1}{2}$.

\begin{proposition}
\label{prop:small_p}
    For any $\p\in[0,\frac{1}{2}]^\Nbb_{\downarrow 0}$ and $n\geq 1$ such that for all $j\geq 1$, one has $p(j)\leq \frac{1}{2nj}$, then
    \begin{equation*}
        \Delta_n^+(\p) \asymp \frac{1}{n}\land \sum_{j\geq 1}p(j).
    \end{equation*}
\end{proposition}

\begin{proof}
    We first give some simple bounds on binomial tails for $q\leq \frac{1}{2n}$. We write $\hat q_n=\frac{Y}{n}$ for $Y\sim\Bcal(n,q)$. For any $k\geq 1$,
    \begin{align*}
        \Pbb\paren{\hat q_n \geq \frac{k}{n}} =\sum_{l=k}^n\binom{n}{l}q^l(1-q)^{n-l} \leq \sum_{l=k}^n \frac{(nq)^l}{l!} \leq 2\frac{(nq)^k}{k!}.
    \end{align*}
    Now let $p\in[0,\frac{1}{2}]^{\Nbb}_{\downarrow 0}$ such that for all $j\geq 1$, $p(j)\leq \frac{1}{2nj}$. We define $U(\p) = \sup_{j\geq 1} njp(j)\leq \frac{1}{2}$. Letting $\hat P_n:= \sup_{j\geq 1}\hat p_n(j)$, for any $k\geq 2$, the union bound implies
    \begin{equation*}
        \Pbb\paren{\hat P_n\geq \frac{k}{n}}\leq \frac{2U(\p)^k}{k!}\sum_{j\geq 1}\frac{1}{j^k}\leq \frac{\pi^2}{3}\frac{U(\p)^k}{k!}.
    \end{equation*}
    Hence,
    \begin{equation*}
        \Ebb\sqb{\hat P_n\1_{\hat P_n\geq \frac{2}{n}}} \leq \sum_{k\geq 2}\frac{k}{n}\Pbb\paren{\hat P_n\geq \frac{k}{n}} \leq \frac{2\pi^2}{3}\frac{U(\p)^2}{n}\leq \frac{\pi^2 U(\p)}{3n}.
    \end{equation*}
    Now let $V(\p) = \sum_{j\geq 1}np(j)$. Note that for any $j\geq 1$, $V(\p)\geq \sum_{i\leq j}np(j)\geq njp(j)$, so that $U(\p)\leq V(\p)\land 1$. By linearity of the expectation, one has
    \begin{equation*}
        \Ebb\sqb{\sum_{j\geq 1}n\hat p_n(j)} = V(\p).
    \end{equation*}
    In particular, since this sum takes integer values and is nonzero whenever $\hat P_n\geq \frac{1}{n}$, we obtain $\Pbb(\hat P_n\geq \frac{1}{n})\leq V(\p)\land 1$. We now show that $\Pbb(\hat P_n\geq \frac{1}{n})\gtrsim V(\p)\land 1$. We have
    \begin{equation*}
        \Pbb\paren{\hat P_n\geq \frac{1}{n}} = 1-\prod_{j\geq 1}(1-p(j))^n
        \geq  1-e^{-V(\p)} \geq c_4 V(\p)\land 1,
    \end{equation*}
    for some universal constant $c_4>0$. Recall that for all $j\geq 1$, one has $p(j)\leq \frac{1}{2n}$. Therefore, whenever $\hat P_n \geq \frac{1}{n}$, we have $\sup_{j\geq 1}\hat p_n(j) - p(j) \geq \frac{1}{2n}$. The previous bound then shows that
    \begin{equation*}
        \Delta_n^+(\p) \geq \frac{c_4}{2} \frac{V(\p)\land 1}{n}.
    \end{equation*}
    On the other hand,
    \begin{align*}
        \Delta_n^+(\p) \leq \Ebb[\hat P_n] &\leq \frac{1}{n} \Pbb\paren{\hat P_n\geq \frac{1}{n}} + \Ebb\sqb{\hat P_n\1_{\hat P_n\geq \frac{2}{n}}}\\
        &\leq \frac{V(\p)\land 1}{n} + \frac{\pi^2 U(\p)}{3n}\\
        &\leq 5\frac{V(\p)\land 1}{n}.
    \end{align*}
    where in the last inequality we used $U(\p)\leq V(\p)\land 1$. This ends the proof of the proposition.
\end{proof}

Using the previous results, we are now ready to prove the complete behavior of $\Delta_n(\p)$.

\vspace{2mm}

\begin{proof}{\textbf{of \cref{thm:main_result}}}
    Propositions \ref{prop:large_p} and \ref{prop:small_p} provide the complete behavior of $\Delta_n^+(\p)$. It remains to show that this is the leading term in the decomposition $\Delta_n(\p)\asymp \Delta_n^+(\p) + \Delta_n^-(\p)$. We first consider the case when $p(j)\leq \frac{1}{2nj}$ for all $j\geq 1$. In that case, we have directly
    \begin{equation*}
        \Delta_n^-(\p) \leq \sup_{j\geq 1} p(j) \leq \frac{1}{n}\land \sum_{j\geq 1}p(j).
    \end{equation*}
    Now suppose that $p(j)\geq \frac{1}{2nj}$ for $j\geq 1$. By construction of $\phi_{J,q}(n)$, one has for all $J\geq 1$ and $q\in(0,\frac{1}{2}]$,
    \begin{equation*}
        \phi_{J,q}(n) \geq 1\land \sqrt{\frac{q\ln(J+1)}{n}},
    \end{equation*}
    i.e. intuitively the second regime is larger than the third. As a result,
    \begin{equation*}
        \sup_{j\geq 1}\phi_{j,p(j)}(n) \geq 1\land \sup_{j\geq 1}\sqrt{\frac{p(j)\ln(j+1)}{n}} = 1\land \sqrt{\frac{S(\p)}{n}}.
    \end{equation*}
    Next, we clearly have $\Delta_n^-(\p)\leq 1$. Also, in the proof of \citet[Theorem 3]{cohen2022local}, the authors show that
    \begin{equation*}
        \Delta_n^-(\p) \leq \sqrt{\frac{S(\p)}{n}}.
    \end{equation*}
    Hence, we finally obtain
    \begin{equation*}
        \Delta_n^+(\p)\asymp\sup_{j\geq 1}\phi_{j,p(j)}(n) \geq 1\land \sqrt{\frac{S(\p)}{n}} \geq \Delta_n^-(\p).
    \end{equation*}
    This ends the proof that $\Delta_n(\p)\asymp\Delta_n^+(\p)$, which implies the desired result.    
\end{proof}

\section{Proof of the localized Dvoretzky-Kiefer-Wolfowitz results}
\label{sec:proof_local_dkw}

We first prove our local DKW result in \cref{thm:local_dkw}. Before considering the case of general distributions and intervals, we focus on the simpler case of the uniform distribution and consider intervals of the form $[q/2,q]$.
\begin{lemma}\label{lemma:uniform_interval}
    Let $X_1,\ldots,X_n\overset{i.i.d.}{\sim}\Ucal([0,1])$ and $F_m$ be the empirical CDF. Let $q\in(0,\frac{1}{2}]$. Then, for any $t>0$,
    \begin{equation*}
        \Pbb\paren{\sup_{x\in[\frac{q}{2},q]} |F_n(x) - F(x)| > t \sqrt{\frac{q}{n}} } \leq c_1 e^{-c_2 \min(t^2,t\sqrt{nq})},
    \end{equation*}
    for some universal constants $c_1,c_2>0$.
\end{lemma}

\begin{proof}
    For the proof, we apply Bernstein inequalities to the number of points falling in intervals within $[\frac{q}{2},q]$. We first treat the simple case when $t\geq\sqrt{nq}$. In that case, Bernstein's inequality shows that
    \begin{equation*}
        \Pbb\paren{F_n(q)-q \geq t\sqrt{\frac{q}{n}}}\leq \exp\paren{-\frac{\frac{1}{2}t^2nq}{nq(1-q)+\frac{t\sqrt{nq}}{3}} } \leq   e^{-t\sqrt{nq}/4}.
    \end{equation*}
    Suppose that the complementary event is met, then for any $x\in[0,q]$, we have $0\leq F_n(x) \leq F_n(q) \leq q + t\sqrt{\frac{q}{n}} \leq 2 t\sqrt{\frac{q}{n}}$. In particular, $|F_n(x) - x| \leq 2t\sqrt{\frac{q}{n}}$. Hence,
    \begin{equation*}
        \Pbb\paren{\sup_{x\leq q}|F_n(x)-x| \geq 2t\sqrt{\frac{q}{n}}} \leq e^{-t\sqrt{nq}/4}.
    \end{equation*}
    This shows that for any $t\geq 2\sqrt{nq}$,
    \begin{equation*}
        \Pbb\paren{\sup_{x\leq q}|F_n(x)-x| \geq t\sqrt{\frac{q}{n}}} \leq e^{-t\sqrt{nq}/8}.
    \end{equation*}

    In the rest of the proof, we suppose that $t\leq \frac{3}{2}\sqrt{nq}$ which implies $3n\frac{q}{2} \geq t\sqrt{nq}$. We recall that if $Z_1,\ldots,Z_n\overset{i.i.d.}{\sim}\Bcal(r)$ are independent Bernoulli variables with $r\in(0,\frac{1}{2}]$, the Bernstein's inequality yields for $\delta\leq 3nr$,
    \begin{equation*}
        \Pbb\paren{\left| \sum_{i=1}^n Z_i - nr\right| \geq \delta} \leq 2\exp\paren{-\frac{\frac{1}{2}\delta^2}{nr(1-r)+\frac{\delta}{3}} } \leq 2\exp\paren{-\frac{\delta^2}{4nr}}.
    \end{equation*}
    Now consider any $u\geq 1$ such that $\frac{3nq}{2^u} \geq t\sqrt{nq}$, and any $v\in \{0,\ldots,2^{u-1}-1\}$. We apply the previous inequality to the points $X_1,\ldots,X_n$ falling in the interval $I_{u,v}:=(\frac{q}{2} + q\frac{v}{2^u},\frac{q}{2} + q\frac{v+1}{2^u}]$. We obtain
    \begin{equation*}
        \Pbb\paren{\left| \frac{1}{n}\sum_{i=1}^n \1(Y_i\in I_{u,v}) - \frac{q}{2^u}\right| \geq \frac{t}{2^{u/3}}\sqrt{\frac{q}{n}}} \leq 2e^{-t^2 2^{u/3-2}}.
    \end{equation*}
    Last, using the same inequality, given that $3n\frac{q}{2} \geq t\sqrt{nq}$, we have that
    \begin{equation*}
        \Pbb\paren{\left|\frac{1}{n}\sum_{i=1}^n\1\paren{Y_i \leq \frac{q}{2}} - \frac{q}{2}\right| \geq t\sqrt{\frac{q}{n}}} \leq 2e^{-t^2/2}.
    \end{equation*}
    We denote by $E_t$ the intersection of the complementary events described above. By the union bound, we have
    \begin{equation*}
        1-\Pbb(E_t) \leq 2e^{-t^2/2}+  \sum_{u\geq 1}2^{u-1} \cdot 2e^{-t^2 2^{u/3-2}}\leq c_1 e^{-c_2 t^2}
    \end{equation*}
    for some universal constants $c_1,c_2>0$. We now suppose that this event is met and aim to prove an upper bound for $|F_n(x)-F(x)|$ for an arbitrary $x\in[\frac{q}{2},q)$. To do so, we first focus on the points of the form $x_v=\frac{q}{2} + \frac{v}{2^{u_0}}$ where $u_0\geq 1$ is the largest integer for which $\frac{3nq}{2^{u_0}} \geq t\sqrt{nq}$. In particular, we have $\frac{q}{2^{u_0}} \leq \frac{2}{3} t\sqrt{\frac{q}{n}}$. We decompose $v$ in binary encoding via $v = \sum_{u\leq u_0}a_{u}2^{u_0-u}$ where $a_{u}\in\{0,1\}$ for $u\in[u_0]$.  Writing $v_u = \sum_{u'\leq u}a_{u'}2^{u-u'}$, we can write
    \begin{equation*}
        nF_n(x_v)=\sum_{i=1}^n \1\paren{Y_i\leq x_v} = \sum_{i=1}^n \1\paren{Y_i\leq \frac{q}{2}} + \sum_{u=1}^{u_0}\sum_{i=1}^n \1(Y_i\in I_{u,v_u}).
    \end{equation*}
    As a result, on $E$, we have for any $v\in\{0,\ldots,2^{u_0}\}$,
    \begin{equation*}
        |F_n(x_v) - x_v| \leq t\sqrt{\frac{q}{n}} + \sum_{u=1}^{u_0} a_u \frac{t}{2^{u/3}}\sqrt{\frac{q}{n}} \leq \frac{t}{1-2^{-1/3}} \sqrt{\frac{q}{n}}.
    \end{equation*}
    Last, let $x\in(\frac{q}{2},q]$. There exists $v\in \{0,\ldots, 2^{u_0}-1\}$ such that $\frac{v}{2^{u_0}} < x \leq \frac{v+1}{2^{u_0}} $. We note that
    \begin{align*}
        |F_n(x) - x| &\leq \max\paren{\left|F_n\paren{\frac{v}{2^{u_0}}} - \frac{v}{2^{u_0}}\right|, \left|F_n\paren{\frac{v+1}{2^{u_0}}} - \frac{v+1}{2^{u_0}}\right|} + \frac{1}{2^{u_0}}\\
        &\leq \paren{\frac{1}{1-2^{-1/3}} + \frac{2}{3}} t \sqrt{\frac{q}{n}}\\
        &\leq 6t \sqrt{\frac{q}{n}}.
    \end{align*}
    Hence, on $E$, we showed that
    \begin{equation*}
        \sup_{x\in[\frac{q}{2},q]}|F_n(x) - x| \leq 6t \sqrt{\frac{q}{n}}.
    \end{equation*}
    Hence, we showed that for any $t\leq \frac{9}{2}\sqrt{nq}$, one has
    \begin{equation*}
        \Pbb\paren{\sup_{x\in[\frac{q}{2},q]}|F_n(x) - x| > t \sqrt{\frac{q}{n}}} \leq c_1 e^{-c_2 t^2},
    \end{equation*}
    for some universal constants $c_1,c_2>0$.
\end{proof}
We are now ready to prove the local DKW bound for intervals of the form $(-\infty,x_0]$.

\vspace{2mm}

\begin{proof}{\textbf{of Corollary \ref{cor:local_dkw}}}
    First, note that if $F(x_0)\geq \frac{1}{2}$, then we can use the classical DKW Theorem \ref{thm:dkw} to obtain the desired bound. We will therefore suppose without loss of generality that $F(x_0)\leq \frac{1}{2}$. We first prove the result for the uniform distribution. Fix $q\in [0,\frac{1}{2}]$. For convenience, let $\delta = t\sqrt{\frac{q}{n}}$. We first suppose that $q\leq \delta$. Then, Lemma \ref{lemma:uniform_interval} implies in particular that
    \begin{equation*}
        \Pbb\paren{ |F_n(q) -q|\leq \delta} \leq c_1 e^{-c_2\min(t^2,t\sqrt{nq})} = c_1 e^{-c_2 t\sqrt{nq}}.
    \end{equation*}
    Note that on the event $|F_n(q)-q|\leq \delta$, we have in particular for all $x\leq q$ that $F_n(x) \leq F_n(q) \leq q+\delta\leq 2\delta$. As a result, for all $x\leq q$, $|F_n(x) -x| \leq 2\delta$. This yields
    \begin{equation*}
        \Pbb\paren{ \sup_{x\leq q}|F_n(x) -x|\leq 2\delta} \leq c_1 e^{-c_2\min(t^2,t\sqrt{nq})} = c_1 e^{-c_2 t\sqrt{nq}}.
    \end{equation*}
    
    We now consider the case when $q\geq \delta$. Similarly as above, if $|F_n(\delta) - \delta| \leq \delta$, then for any $x\leq \delta$, we have $\sup_{x\leq \delta} |F_n(x) -x| \leq 2\delta$. As a result, we can focus on the interval $[\delta,q]$. We decompose the supremum on intervals of the form $[\frac{q}{2^{u+1}},\frac{q}{2^u}]$ for $u\geq 0$. From the above arguments, it suffices to consider intervals $[\frac{q}{2^{u+1}},\frac{q}{2^u}]$ for $u\leq u_0$ such that $\frac{q}{2^{u_0+1}} \leq t\sqrt{\frac{q}{n}} \leq \frac{q}{2^{u_0}}$. We note that for $0\leq u\leq u_0$, one has $2^u t \leq \sqrt{nq}$, so that $\min(2^u t^2,t\sqrt{nq})  = 2^u t^2$. Hence, by Lemma \ref{lemma:uniform_interval},
    \begin{align*}
        \Pbb\paren{\sup_{x\in[0,q]}|F_n(x) - x| > 2\delta} &\leq \Pbb\paren{\sup_{x\in[\delta,q]}|F_n(x) - x| > \delta} \\
        &\leq \sum_{u=0}^{u_0} \Pbb\paren{\sup_{x\in[\frac{q}{2^{u+1}},\frac{q}{2^u}]}|F_n(x) - x| > 2^{u/2} t \sqrt{\frac{q}{2^u n}}}\\
        &\leq c_1 \sum_{u=0}^{u_0} e^{-c_2 \min(2^u t^2,t\sqrt{nq})}\\
        &= c_1 \sum_{u=0}^{u_0} e^{-c_2 2^u t^2} \leq c_3 e^{-c_4 t^2},
    \end{align*}
    for some universal constants $c_3,c_4>0$. This shows that for some constants $c_3,c_4>0$, we have
    \begin{equation*}
        \Pbb\paren{\sup_{x\in[0,q]}|F_n(x) - x| > 2t\sqrt{\frac{q}{n}}} \leq c_3 e^{-c_4 \min(t^2,t\sqrt{nq})}.
    \end{equation*}
    
    Changing the constants appropriately ends the proof of the theorem for the uniform distribution. The result extends directly to general distributions via a change of variables. Consider a real-valued distribution $\mu$ with CDF $F_X$. If $U\sim\Ucal([0,1])$ is uniform, then $X = F^{-1}(U)\sim \mu$, where we define $F^{-1}(u) = \inf\{x: F(x) \geq u\}$. Because the CDF $F$ is right-continuous, we have in particular $F(F^{-1}(u)) \geq u$. Hence, $F(x) \geq u$ i.if $x \geq F^{-1}(u)$. Given $n$ samples $U_1,\ldots,U_n\overset{i.i.d.}{\sim}\Ucal([0,1])$, we denote by $F_{n,U}$ their empirical CDF. Similarly, letting $X_i = F^{-1}(U_i)$ for $i\in [n]$, we denote by $F_{n,X}$ their empirical CDF. Now note that for any $x\in \Rbb$,
    \begin{equation*}
        F_{n,X}(x)-F(x) = \frac{1}{n}\sum_{i=1}^n \1(F^{-1}(U_i) \leq x) - F(x) = \frac{1}{n}\sum_{i=1}^n \1(U_i \leq F(x)) - F(x).
    \end{equation*}
    As a result, we have
    \begin{equation*}
        \Pbb\paren{\sup_{x\leq x_0}|F_{n,X}(x) - F(x)| > t\sqrt{\frac{F(x_0)}{n}}} \leq \Pbb\paren{\sup_{x\in[0,F(x_0)]}|F_{n,U}(x) - x| > t\sqrt{\frac{F(x_0)}{n}}}.
    \end{equation*}
    This ends the proof of the theorem.
\end{proof}

Last, we now prove the main localized result for intervals $[x_0,x_1]$.

\vspace{2mm}

\begin{proof}{\textbf{of \cref{thm:local_dkw}}}
    If $F(x_0)\leq \frac{1}{2}\leq F(x_1)$, the standard DWK \cref{thm:dkw} gives the desired result. Otherwise, without loss of generality, we suppose that $F(x_1)\leq \frac{1}{2}$. Then, $F(x_1)\geq V =\max_{x\in[x_0,x_1]}F(x)(1-F(x)) \geq \frac{F(x_1)}{2}$. Hence, using Corollary \ref{cor:local_dkw},
    \begin{multline*}
        \Pbb\paren{\sup_{x\in[x_0,x_1]}|F_n(x) - F(x)| >t\sqrt{\frac{V}{n}}} \leq \Pbb\paren{\sup_{x\leq x_1}|F_n(x) - F(x)| >t\sqrt{\frac{F(x_1)}{2n}}} \\
        \leq c_1 e^{-\frac{c_2}{4} \min(t^2,2t\sqrt{nF(x_1)})} \leq c_1 e^{-\frac{c_2}{4} \min(t^2,2t\sqrt{nV})}.
    \end{multline*}
    This ends the proof.
\end{proof}

\section{Proofs of the results on the expected maximum deviation of distributions on $[0,1]$}
\label{sec:proof_continuous01}

Before proving the main result Corollary \ref{cor:continuous[0,1]}, we recall Bennett's inequality.

\begin{lemma}[Bennett's inequality]
\label{lemma:bennett}
    Let $X_1, X_2, \dots, X_n$ be i.i.d.\ random variables with mean $\mu$, variance $\sigma^2$ and $0 \leq X_i \leq 1$ almost surely. Then for any $t>0$ we have

    \begin{equation*}
        \Pbb\paren{\frac{1}{n}\sum_{i=1}^n (X_i-\mu) \geq t}\leq e^{-n \sigma^2\h\left(\frac{t}{\sigma^2}\right)},
    \end{equation*}
    and 
        \begin{equation*}
        \Pbb\paren{\frac{1}{n}\sum_{i=1}^n (X_i-\mu) \geq t}\leq e^{-n \mu\h\left(\frac{t}{\mu}\right)},
    \end{equation*}
    where $\h(u) = (1+u)\ln(1+u)-u$.
\end{lemma}
\begin{proof}
    For the first part, see~\cite[Theorem 2.9]{boucheron2013concentration}. For the second, note that 
    \[
    0 \leq \Ebb[X(1-X)] =\Ebb[X] -\Ebb[X^2] = \mu - \sigma^2 +\mu^2.
    \]
    Hence, $\sigma^2 \leq \mu - \mu^2 \leq \mu$ and $f(x)= x\h\left( \frac t x \right)$ is decreasing on $x \in [0,\infty)$ since $f'(x) = \ln(1+\frac{t}{x}) - \frac t x  \leq 0 $. Thus, the second statement is weaker than the first one.
\end{proof}

We will use this inequality instead of the Chernoff bound (Lemma~\ref{lemma:chernoff_bound}) that we used in the case of Bernoulli random variables.

\vspace{2mm}

\begin{proof}{\textbf{of Corollary \ref{cor:continuous[0,1]}}}
    As a first step, we show that we can upper bound $\Delta_n^+(\mu)$ similarly as in Proposition~\ref{prop:reduction}, by replacing $p(i)$ with $\sigma^2(i)$. 
    
    As in the proof of Proposition~\ref{prop:reduction}, let $\varepsilon:=\sup_{i\geq 1} \varepsilon_{i,\sigma^2(i)}(n)$ and $\tilde\varepsilon = (\varepsilon\land\frac{1}{4})\lor\frac{1}{n}$. The same proof shows that for any $i\geq 1$,
    \begin{equation*}
        D(\sigma^2(i)+\tilde\varepsilon\parallel \sigma^2(i)) \geq \frac{\ln(2i)}{nC},
    \end{equation*}
    for some universal constant $C>0$. The proof also shows that $\varepsilon\geq \frac{1}{4n}$, hence $\varepsilon\asymp\tilde \varepsilon$. Next, by Lemma~\ref{lemma:compare_kl_h}, we have
    \begin{equation*}
        \sigma^2(i) \h\paren{\frac{2\tilde \varepsilon}{\sigma^2(i)}} \geq D(\sigma^2(i) +\tilde\varepsilon \parallel\sigma^2(i)) \geq \frac{\ln(2i)}{nC}
    \end{equation*}
    where in the last inequality, we used Lemma~\ref{lemma:compare_kl_h}. We now use Bennett's inequality from Lemma~\ref{lemma:bennett} instead together with the convexity of the function $\sigma^2(i)\h(\frac{\cdot}{\sigma^2(i)})$ , to obtain as in the proof of Proposition~\ref{prop:reduction} that for any $k\geq 1$,
    \begin{equation*}
        \Pbb(\hat p_n(i)\geq p(i) + 2kC\varepsilon)\leq e^{-\sigma^2(i)\h\paren{\frac{2kC\varepsilon}{\sigma^2(i)}}} \leq e^{-kC\sigma^2(i)\h\paren{\frac{2\varepsilon}{\sigma^2(i)}}} \leq \frac{1}{(2i)^k}.
    \end{equation*}
    The same union bound argument then shows that $\Delta_n^+(\mu) \leq 3C\varepsilon\lesssim\varepsilon.$

    In summary, this shows that if there exists $j\geq 1$ such that $\sigma^2(j) \geq \frac{1}{2nj}$, then from Proposition~\ref{prop:lower_bound} one has $\varepsilon=\sup_{i\geq 1}\varepsilon_{i,\sigma^2(i)}(n)\geq 0$. As a result, the proof of Proposition \ref{prop:large_p} shows that $\sup_{i\geq 1}\varepsilon_{i,\sigma^2(i)}(n)\asymp \sup_{i\geq 1}\phi_{i,\sigma^2(i)}(n)$, which gives
    \begin{equation*}\label{eq:delta_+}
        \Delta_n^+(\mu)\lesssim \sup_{i\geq 1}\varepsilon_{i,\sigma^2(i)}(n) \asymp 1\land \sup_{j\geq 1} \paren{\sqrt{\frac{\sigma^2(j)\ln(j+1)}{n}} \lor  \frac{\ln(j+1)}{n\ln\paren{2+\frac{\ln(j+1)}{n\sigma^2(j)}}}  }.
    \end{equation*}
    It now remains to bound $\Delta_n^-(\mu)$. This can be done in a completely symmetric manner, by considering the distribution $\tilde\mu$ of $(1-X_i)_{i\geq 1}$ for $X\sim\mu$. We obtain directly
    \begin{equation*}
        \Delta_n^-(\mu) = \Delta_n^+(\tilde\mu) \lesssim \sup_{i\geq 1}\varepsilon_{i,\sigma^2(i)}(n),
    \end{equation*}
    where in the last inequality, we applied Eq~\eqref{eq:delta_+} to $\tilde\mu$. Finally, we showed that if there exists $j\geq 1$ such that $\sigma^2(j) \geq \frac{1}{2nj}$, then
    \begin{equation*}
        \Delta_n(\mu)\lesssim 1\land \sup_{j\geq 1} \paren{\sqrt{\frac{\sigma^2(j)\ln(j+1)}{n}} \lor  \frac{\ln(j+1)}{n\ln\paren{2+\frac{\ln(j+1)}{n\sigma^2(j)}}}  }.
    \end{equation*}

    We now suppose that for all $j\geq 1$, one has $p(j)\leq \frac{1}{2nj}$. Fix $i\geq 1$. Since $\sigma^2(i)\leq \frac{1}{2ni}$, for any $k\geq 1$, one has $\frac{4k}{n} \geq 8\sigma^2(i)$. Then, Bennett's inequality in Lemma~\ref{lemma:bennett} together with a lower bound from Lemma~\ref{lemma:estimates_kl} shows that
    \begin{equation*}
        \Pbb\paren{\hat p_n(i)\geq p(i) + \frac{4k}{n}} \leq e^{-\sigma^2(i)\h\paren{\frac{4k}{n\sigma^2(i)}}} \leq \paren{\frac{n\sigma^2(i)}{4k}}^{\frac{4k}{n\sigma^2(i)}}\leq \frac{1}{(8ki)^{8ki}}.
    \end{equation*}
    As a result, the same computations as in the proof of Proposition~\ref{prop:reduction} show that
    \begin{equation*}
        \Delta_n^+(\mu)\leq \frac{12}{n}.
    \end{equation*}
    As before, the argument is symmetric, hence we obtain $\Delta_n^-(\mu)\leq \frac{12}{n}$ as well. This shows that
    \begin{equation*}
        \Delta_n(\mu)\lesssim\frac{1}{n}.
    \end{equation*}
    Next, for any $i\geq 1$, note that $\text{Var}(\hat p_n(i)) = \frac{\sigma^2(i)}{n}.$ As a result, for any $c>0$, Chebyshev's inequality yields
    \begin{equation*}
        \Pbb(|\hat p_n(i) -p(i)|\geq c) \leq \frac{\sigma^2(i)}{nc^2}.
    \end{equation*}
    Hence, by the union bound,
    \begin{equation*}
        \Pbb(\|\hat p_n -p\|_\infty\geq c)\leq \frac{1}{nc^2}\sum_{i\geq 1}\sigma^2(i).
    \end{equation*}
    Now suppose that $\sum_{i\geq 1}\sigma^2(i)<\infty$. For simplicity, let $\eta = \sqrt{\frac{\sum_{i\geq 1} \sigma^2(i)}{n}}$. Then,
    \begin{align*}
        \Delta_n(\mu) = \Ebb\|\hat p_n-p\|_\infty &\leq \eta   + \sum_{k\geq 1} 2^k\eta  \Pbb\paren{\|\hat p_n -p\|_\infty\geq 2^{k-1}\eta}\\
        &\leq \eta + \eta\sum_{k\geq 1} \frac{\sum_{i\geq 1}\sigma^2(i)}{n 2^{k-2}\eta^2} = 5\eta.
    \end{align*}
    This ends the proof that
    \begin{equation*}
        \Delta_n(\mu)\lesssim\frac{1}{n}\land  \sqrt{\frac{\sum_{i\geq 1} \sigma^2(i)}{n}},
    \end{equation*}
    which ends the proof of the result.
\end{proof}

By \cref{thm:main_result}, we know that the upper bounds from Corollary~\ref{cor:continuous[0,1]} are attained using a sequence of independent Bernoulli random variables---we recall that in this case, since $p(i)\in[0,1/2]$, for $i\geq 1$, one has $\sigma^2(i)\asymp p(i)$---except in the case when $\sum_{j\geq 1}\sigma^2(j)\leq \frac{1}{2n}$. 

In that case, changing the support from $\{0,1\}$ to $\{0,\sqrt{2n\sum_{j\geq 1}\sigma^2(j)}\}$ achieves the desired upper bound. For convenience, define $\eta = \sqrt{2n\sum_{i\geq 1} \sigma^2(i)}$. We consider the sequence of independent variables $X_i = \eta Z_i$ where $(Z_i)_{i\geq 1}$ are independent Bernoulli variables with parameters $\frac{\sigma^2(i)}{\eta^2}$. We denote by $\mu$ this distribution. We have that
\begin{equation*}
    \Pbb\paren{\hat p_n(i) \geq \frac{\eta}{n} } = 1-\prod_{j\geq 1}\paren{1-\frac{\sigma^2(i)}{\eta^2} }^n \geq 1-\exp\paren{-\frac{n}{\eta^2}\sum_{j\geq 1}\sigma^2(i) }  \geq 1-e^{-1/2}.
\end{equation*}
As a result, since $p(i) =\eta\frac{\sigma^2(i)}{\eta^2} \leq \frac{\eta}{2n}$, we obtained that
\begin{equation*}
    \Ebb \|\hat p_n-p\|_\infty \gtrsim \frac{\eta}{n} \asymp \sqrt{\frac{\sum_{i\geq 1}\sigma^2(i)}{n}}.
\end{equation*}

\section{Proofs of the results on expected empirical deviations in $\ell^q$ norms}
\label{sec:lq_norms}

We first prove the convergence characterization from Proposition~\ref{prop:characterization_lq_convergence}.

\vspace{2mm}

\begin{proof}{\textbf{of Proposition~\ref{prop:characterization_lq_convergence}}}
    Suppose that $\|\p\|_1=\infty$. 
    Further let  $p(i)\to 0$ as $i\to\infty$; otherwise the result would be straightforward. 
    Let $(X_i)_{i\geq 1}$ be a sequence of independent Bernoulli random variables such that $X_i\sim\Bcal(p(i))$. Because $\sum_{j\geq 1}p(j)=\infty$, by Borel-Cantelli's lemma, almost surely, there is an infinite number of indices $i\geq 1$ for which $X_i=1$. In particular, with full probability, there is an infinite number of indices $i$ for which $\hat p_n(i)\geq \frac{1}{n}$ and thus infinitely many for which $|\hat p_n(i) - p(i)|\geq \frac{1}{2n}$ since $p(i)\to 0$ as $i\to\infty$. As a result, $\|\hat\p_n-\p\|_q=\infty \;(a.s.)$.

    Now suppose that $\|\p\|_1<\infty$. We note that
    \begin{equation*}
        \Ebb\|\hat \p_n-\p\|_q \leq \Ebb \|\hat \p_n-\p\|_1 \asymp \sum_{j\geq 1} p(j)\land \sqrt{\frac{p(j)}{n}}.
    \end{equation*}
    In particular, for any $\varepsilon>0$, there exists $i\geq 1$ such that $\sum_{j\geq i}p(j)<\varepsilon$. Then, for $n\geq \frac{1}{p(i)}$, we have
    \begin{equation*}
        \Ebb\|\hat \p_n-\p\|_q \lesssim  \frac{1}{\sqrt n}\sum_{j<i}\sqrt{p(j)} + \sum_{j\geq i}p(j) \leq \frac{1}{\sqrt n}\sum_{j<i}\sqrt{p(j)} + \varepsilon.
    \end{equation*}
    Hence, $\limsup_{n\to\infty} \Ebb\|\hat \p_n-\p\|_q\leq \varepsilon$. Because this holds for any $\varepsilon$, this shows that $\Ebb\|\hat \p_n-\p\|_q\to 0$ as $n\to\infty$.
\end{proof}

We now provide bounds on the deviation $\Ebb \|\hat \p_n-\p\|_q$ when $\|\p\|_1<\infty$. To do so, we first need estimates on the central moments of binomials.

\begin{lemma}\label{lemma:central_moments}
    Let $n\geq 1$ and $0\leq p\leq \frac{1}{2}$. Let $Y\sim\Bcal(n,p)$ be a binomial and $q\geq 1$. Then,
    \begin{equation*}
        \Ebb |Y-np|^q \asymp_q \psi_q(n,p):=
        \begin{cases}
            (npq)^{q/2} & p\geq \frac{q}{2n}\\
            \paren{\frac{q}{\ln\frac{q}{np}}}^q & \frac{q}{ne^q} \leq p\leq \frac{q}{2n} \\
            np & p\leq \frac{q}{ne^q}
        \end{cases}
    \end{equation*}
    where the $\asymp_q$ term hides factors $\Omega(c^q)$ and $\mathcal O(C^q)$ for universal constants $c,C>0$.
\end{lemma}

\begin{proof}
We consider the three different regimes separately. Before doing so, we introduce some notations. For convenience, we will use the extended factorials to define for any $l\in[-np,n(1-p)]$,
\begin{equation*}
    b_l := \binom{n}{np+l} p^{np+l} (1-p)^{n(1-p)-l}  l^q.
\end{equation*}

\paragraph{Regime 1: $\frac{q}{2n}\leq p \leq \frac{1}{2} $.} We aim to understand the sequence $(b_l)_l$ and start with the right tails when $l\geq 0$. Let $l\geq 2\sqrt{npq}$. We can use the convexity inequality $e^x \leq 1+3x$ for $x\in[0,1]$ to obtain
\begin{equation}\label{eq:ratio_upper_bound}
    \frac{b_{l+1}}{b_l} = \frac{n(1-p)-l}{np+l+1} \cdot \frac{p}{1-p}\paren{1+\frac{1}{l}}^q \leq \frac{e^{q/l}}{1+\frac{l}{np}} \leq \frac{e^{\sqrt{q/np}/2}}{1+2\sqrt{q/np}} \leq e^{-\sqrt{q/np}/6} \leq e^{-1/6}.
\end{equation}
In the second-to-last inequality, we used the fact that $\frac{2}{3}\sqrt{\frac{q}{np}}\leq \frac{2\sqrt 2}{3}\leq 1$. In particular, this shows that if $k_1 = \ceil{np+2\sqrt{npq}}$,
\begin{equation}\label{eq:case3_large_deviations}
    \sum_{k=k_1}^n b_{k-np} \asymp b_{k_1-np}.
\end{equation}
On the other hand, if $1\leq l\leq \frac{1}{6}\sqrt{npq}$
\begin{equation*}
    \frac{b_{l+1}}{b_l} \geq \paren{1-\frac{3l}{np}} \paren{1+\frac{q}{l}}  \geq \paren{1-\frac{1}{2}\sqrt{\frac{q}{np}}}  \paren{1+6\sqrt{\frac{q}{np}}}\leq 1
\end{equation*}
In the last inequality, we used the fact that $\sqrt{\frac{q}{np}}\leq \sqrt 2$. As a result, the maximum of $b_l$ for $l\geq 0$ is achieved for $l\in[\frac{1}{6}\sqrt{npq},2\sqrt{npq}]$, and if $k_2 = \ceil{np+\frac{1}{6}\sqrt{npq}}$, we obtained
\begin{equation}\label{eq:case3_small_deviations}
    \sum_{k=\ceil{np}}^{k_2-1} b_{k-np} \leq \frac{1}{6}\sqrt{npq} \cdot b_{k_2-np}.
\end{equation}
We next show that up to exponential terms in $q$, $b_l$ has same order within this range. Precisely, for $l\in [\frac{1}{6}\sqrt{npq},2\sqrt{npq}]$ and an integer $0\leq r\leq 2\sqrt{npq}$, we have
\begin{equation*}
    \frac{b_{l+r}}{b_l} \leq \paren{1+\frac{r}{l}}^q \leq \paren{1+\frac{1}{8}}^q\asymp_q 1.
\end{equation*}
We now turn to the lower bound and now suppose $n\geq 100$q. We will treat the other case separately. We recall that $l+r\leq 4\sqrt{npq} \leq \frac{4n}{\sqrt{200}} \leq \frac{n}{4}$, so that $\frac{l+r-1}{n(1-p)}\leq \frac{1}{2}$. Next, by convexity, we have the inequality $1-x\geq e^{-2x}$ for $x\in[0,1/2]$. Hence,
\begin{equation*}
    \frac{b_{l+r}}{b_l} \geq \paren{\frac{1-\frac{l+r-1}{n(1-p)}}{1+\frac{l+r}{np}}}^r \geq \exp\sqb{-r\paren{8\sqrt{\frac{q}{n(1-p)}} + 4\sqrt{\frac{q}{np}}}} \geq e^{-2(8\sqrt 2+4)q}\asymp_q 1.
\end{equation*}
Hence, this shows that $b_l\asymp_q b_{l+r}$. In particular,  the two previous statements showed that
\begin{equation*}
    \sum_{k=k_2}^{k_1} b_{k-np} \asymp_q \sqrt{npq} \cdot b_{k_1-np}.
\end{equation*}
Combining the previous equation together with Eq~\eqref{eq:case3_large_deviations} and \eqref{eq:case3_small_deviations}, we have,
\begin{equation*}
    \Ebb \sqb{|Y-np|^q \1_{Y>np}} = \sum_{np\leq k\leq n} b_{k-np}\asymp_q \sqrt{npq} \cdot b_{k_1-np}.
\end{equation*}
We then use Stirling's approximation formula to estimate the right-hand side. Noting that $1\leq \sqrt{npq} \leq \frac{np}{10}$ since $n\geq 100q$, we have
\begin{equation*}
    b_{k_1-np} \asymp_q \frac{(\sqrt{npq})^q}{\sqrt{np}\paren{\frac{k_1}{np}}^{k_1} \paren{\frac{1-\frac{k_1}{n}}{1-p}}^{n-k_1}}.  
\end{equation*}
Now writing $k_1=np+l_1$, we have that $\frac{l_1}{np}\leq 2\sqrt{\frac{q}{np}}\leq 2$ and $\frac{k_1 l_1^2}{(np)^2}=\Ocal(q)$. Then,
\begin{equation*}
    \paren{\frac{k_1}{np}}^{k_1} = \exp\paren{k_1\ln\paren{1+\frac{l_1}{np}}}  =   \exp\paren{\frac{k_1 l_1}{np} + k_1\Ocal\paren{\frac{l_1^2}{(np)^2}}} = e^{l_1 + \Ocal(q)}\asymp_q e^{l_1}.
\end{equation*}
Similarly,
\begin{align*}
    \paren{\frac{1-\frac{k_1}{n}}{1-p}}^{n-k_1} &= \exp\paren{(n-k_1)\ln\paren{1-\frac{l_1}{n(1-p)}}} \\
    &= \exp\paren{-l_1-\frac{l_1^2}{n(1-p)}+(n-k_1)\Ocal\paren{\frac{l_1^2}{n^2}}} \asymp_q e^{-l_1}.
\end{align*}
As a result, combining all the previous estimates gives
\begin{equation*}
    \Ebb \sqb{|Y-np|^q \1_{Y>np}} \asymp_q\sqrt q\cdot (npq)^{q/2}\asymp_q (npq)^{q/2}.
\end{equation*}
We now turn to the left tails. For $2\sqrt{npq}\leq l \leq np$,
\begin{equation*}
    \frac{b_{-l-1}}{b_{-l}} = \frac{1-\frac{l}{np}}{1+\frac{l+1}{n(1-p)}} \paren{1+\frac{1}{l}}^q \leq \exp\paren{\frac{q}{l}-\frac{l}{np}} \leq \exp\paren{-\sqrt{\frac{q}{np}}} \leq \frac{1}{e}.
\end{equation*}
Thus, if $k_3 = \floor{np-2\sqrt{npq}} \geq 0$, we have that $\sum_{k=0}^{k_3}b_{k-np}\asymp b_{k_3-np}$. It suffices then to focus on the terms $b_{-l}$ for $0\leq l\leq 2\sqrt{npq}+1$. Going back to the previous displayed equation shows that the term $\binom{n}{np-l}p^{np-l}(1-p)^{n(1-p)+l}$ is decreasing with $l\geq 0$. As a result, for any $k_3\leq k\leq \floor{np}$, we have
\begin{align*}
    b_{k-np}\leq \binom{n}{\floor{np}}p^{\floor{np}}(1-p)^{n-\floor{np}} (2\sqrt{npq}+1)^q  \asymp_q \frac{(npq)^{q/2}}{\sqrt{np}}. 
\end{align*}
In the last inequality, we used Stirling's approximation formula. As a result, $\sum_{k=k_3}^{\floor{np}}b_{k-np}\lesssim_q \sqrt q\cdot (npq)^{q/2}\asymp (npq)^{q/2}$. Combining the previous equations shows that for $n\geq 100q$,
\begin{equation*}
    \Ebb|Y-np|^q = \sum_{k=0}^n b_{k-np} \asymp_q (npq)^{q/2}.
\end{equation*}

We now treat the case $n\leq 100q$. In that case, $\frac{1}{100}\leq p\leq \frac{1}{2}$ so that $2^n, p^n\asymp_q 1$. Hence,
\begin{equation*}
    \Ebb |Y-np|^q \asymp_q\sum_{k=0}^n |k-np|^q \asymp_q q^q \asymp_q (npq)^{q/2}.
\end{equation*}

\paragraph{Regime 2: $\frac{q}{ne^q}\leq p\leq \frac{q}{2n}$.} Again, we start with the right tails. For convenience, we let
\begin{equation*}
    L:= \frac{q}{\ln\frac{q}{np}}.
\end{equation*}
We note that $L\gtrsim np$. Using similar computations as in Eq~\eqref{eq:ratio_upper_bound}, for $l\geq L$, we have
\begin{equation*}
    \frac{b_{l+1}}{b_l} \leq \frac{np}{l} e^{q/l} \leq \frac{np}{L} e^{q/L} = \sqrt{\frac{np}{q}}\ln\frac{q}{np} \leq \frac{2}{e}.
\end{equation*}
Hence, after $l=L$, the decay of $b_l$ is exponential. Hence, $\sum_{k=\ceil{np+L}}^n b_{k-np}\asymp b_{\ceil{np+L}-np}$. This also shows that if $k_{max}$ is the integer for which $b_{k_{max}-np}$ is maximized and $k_{max}\geq np$, we have $k_{max}-np\leq L$. As a result
\begin{multline*}
     b_{k_{max}-np}\leq \Ebb \sqb{|Y-np|^q \1_{Y>np}} = \sum_{np\leq k\leq n}b_{k-np}\\
     \leq (1+L)b_{k_{max}-np} + \sum_{k\geq np+L} b_{k-np} \asymp L b_{k_{max}-np}.
\end{multline*}
Now note that $1\lesssim L\lesssim q$, so that $L\asymp_q 1$. Thus, with $l_{max} = k_{max}-np$,
\begin{equation*}
    \Ebb \sqb{|Y-np|^q \1_{Y>np}} \asymp_q b_{l_{max}}.
\end{equation*}
We recall that we already know $l_{max}\leq L$. Now for any $l\in[0,L]$, by Stirling's approximation formula,
\begin{equation*}
    b_l\asymp \frac{1}{\sqrt{np}} \frac{l^q}{\paren{1+\frac{l}{np}}^{np+l} \paren{1-\frac{l}{n(1-p)}}^{n(1-p)-l}}
\end{equation*}
Now $\frac{q}{e^q}\leq np\leq q$ so that $\sqrt{np}\asymp_q 1$. Also, since $l\leq L\lesssim q\leq n$, we have
\begin{equation*}
     1\geq \paren{1-\frac{l}{n(1-p)}}^{n(1-p)-l} \geq \paren{1-\frac{2L}{n}}^n = e^{-\Ocal(L)}\asymp_q 1.
\end{equation*}
Next, $1\leq \paren{1+\frac{l}{np}}^{np} \leq e^l -e^{\Ocal(q)}\asymp_q 1$. As a result, for $l\in[0,L]$,
\begin{equation*}
    b_l \asymp_q \frac{l^q}{\paren{1+\frac{l}{np}}^l}.
\end{equation*}
Now if $l\leq np$, we have $b_l\lesssim_q (np)^q\lesssim_q L^q$. On the other hand, if $l\gtrsim np$,
\begin{equation*}
    \paren{1+\frac{l}{np}}^l = \exp\paren{l\Ocal\paren{\ln \frac{l}{np}}}\leq \exp\paren{L\Ocal\paren{\ln \frac{q}{np}}} = e^{\Ocal(q)}\asymp_q 1.
\end{equation*}
Hence, we obtained that for $np\lesssim l\leq L$, $b_l\asymp_q L^q$, while for $0\leq l\leq np$ $b_l\lesssim_q L^q$. As a result we obtained
\begin{equation*}
    \Ebb \sqb{|Y-np|^q \1_{Y>np}} \asymp_q b_{l_{max}} \asymp L^q.
\end{equation*}
The left tail bound is immediate since 
\begin{equation*}
    \Ebb \sqb{|Y-np|^q \1_{Y<np}} \leq (np)^q\lesssim_q L^q.
\end{equation*}
Combining the two previous equations gives the desired result $\Ebb |Y-np|^q\asymp_q L^q$.

\paragraph{Regime 3: $p\leq \frac{q}{ne^q}$.}
In particular, $p\leq \frac{1}{2n}$ so that $(1-p)^n\asymp 1$. Hence, noting that $|k-np|\asymp k$ for any $k\geq 1$, we obtain
\begin{equation}\label{eq:case_1}
    \Ebb|Y-np|^q \asymp_q np + \sum_{k=2}^n \binom{n}{k} p^k k^q.
\end{equation}
Now note that
\begin{equation*}
    \sum_{k=2}^n \binom{n}{k} p^k k^q \leq \sum_{k\geq 2} \frac{(np)^k}{k!}k^q
\end{equation*}
Letting $a_k := \frac{(np)^k}{k!}k^q$, we have for any $k\geq 2$
\begin{equation*}
    \frac{a_{k+1}}{a_k} = \frac{np}{k+1}\paren{1+\frac{1}{k}}^q \leq \frac{qe^{q/k}}{ke^q} \leq \frac{q}{2e^{q/2}} \leq \frac{1}{e}.
\end{equation*}
As a result,
\begin{equation*}
    \sum_{k=2}^n \binom{n}{k} p^k k^q  \leq \frac{e}{2(e-1)}(np)^2 2^q.
\end{equation*}
Plugging this into Eq~\eqref{eq:case_1} yields
\begin{equation*}
    \Ebb|Y-np|^q \asymp_q np.
\end{equation*}

This ends the proof of the lemma.
\end{proof}

We are now ready to prove the following result, which gives general bounds on $\Ebb \|\hat \p_n-\p\|_q$ as well the asymptotic convergence rate when $q\geq 2$ up to a factor $\Theta(\sqrt q)$.

\begin{proposition}\label{prop:lq_bounds}
    Let $\p\in[0,\frac{1}{2}]^\Nbb_{\downarrow 0}$ such that $\|\p\|_1<\infty$. For $q\geq 1$ and $n\geq 1$, we have
    \begin{equation*}
        \frac{1}{\sqrt n} \paren{\sum_{p(i) \geq \frac{1}{n} }p(i)^{q/2}}^{1/q} + \paren{\sum_{p(i)\leq \frac{1}{n}}p(i)^q}^{1/q} \lesssim \Ebb \|\hat \p_n-\p\|_q \lesssim \paren{\frac{1}{n^q} \sum_{i\geq 1} \psi_q(n,p(i))}^{1/q}.
    \end{equation*}
\end{proposition}

\begin{proof}
    We start by observing that by Jensen's inequality, one has
    \begin{equation}\label{eq:jensens}
        \paren{\sum_{i\geq 1}(\Ebb|\hat p_n(i)-p(i)|)^q}^{1/q} \leq \Ebb \|\hat \p_n-\p\|_q \leq (\Ebb \|\hat \p_n-\p\|_q^q)^{1/q}.
    \end{equation}
    The right-hand side inequality uses the convexity of $x\geq 0\mapsto x^q$ and the left-hand side uses the convexity of $x\geq 0\mapsto (c+x^q)^{1/q}$ for any fixed $c\geq 0$. Now using \cref{lemma:central_moments}, we have
    \begin{equation*}
        (\Ebb \|\hat \p_n-\p\|_q^q)^{1/q} \asymp \paren{\frac{1}{n^q}\sum_{i\geq 1} \psi_q(n,p(i))}^{1/q},
    \end{equation*}
    which gives the desired upper bound. Next,
    \begin{align*}
        \paren{\sum_{i\geq 1}(\Ebb|\hat p_n(i)-p(i)|)^q}^{1/q} &\asymp \paren{\sum_{i\geq 1} \paren{\sqrt{\frac{p(i)}{n} }\lor p(i)}^q}^{1/q}\\
    &\asymp \frac{1}{\sqrt n} \paren{\sum_{p(i) \geq \frac{1}{n} }p(i)^{q/2}}^{1/q} + \paren{\sum_{p(i)\leq \frac{1}{n}}p(i)^q}^{1/q}.
    \end{align*}
    This ends the proof of the proposition.
\end{proof}

We are now ready to prove the asymptotic bounds from Proposition~\ref{prop:lq_asymptotic} using the previous result.

\vspace{2mm}

\begin{proof}{\textbf{of Proposition~\ref{prop:lq_asymptotic}}}
    We now analyze the asymptotic convergence of $\Ebb\|\hat \p_n-\p\|_q$ when $q\geq 2$. Since $q\geq 2$, we have $\|\p\|_{q/2}^{q/2}\leq \|\p\|_1<\infty$. For any $i\geq 1$, for $n\geq 1/p(i)$, the previous result shows that
    \begin{equation*}
          \Ebb \|\hat \p_n-\p\|_q \gtrsim \frac{1}{\sqrt n} \paren{\sum_{j\leq i} p(j)^{q/2}}^{1/q}.
    \end{equation*}
    Because this holds for any $i\geq 1$, we obtain the desired lower bound
    \begin{equation*}
        \liminf_{n\to\infty} \sqrt n \Ebb \|\hat\p_n-\p\|_q \gtrsim \paren{\sum_{j\geq 1} p(j)^{q/2}}^{1/q} = \sqrt{\|\p\|_{q/2}}.
    \end{equation*}
    For the upper bound, we first simplify the characterization of the central moments of binomials given in Lemma~\ref{lemma:central_moments}. We obtain directly
    \begin{equation*}
        \psi_q(n,p) \lesssim_q
        \begin{cases}
            (npq)^{q/2} & p\geq \frac{q}{2n}\\
            np q^q& p\leq \frac{q}{2n},
        \end{cases}
    \end{equation*}
    where $\lesssim_q$ hides factors $\Ocal(C^q)$ for a universal constant $C>0$. Here we only simplified the second regime for which $np\leq 1$. As a result, we have that
    \begin{align*}
        (\Ebb \|\hat \p_n-\p\|_q^q)^{1/q} &\lesssim  \sqrt{\frac{q}{n}} \paren{\sum_{p(i) \geq \frac{q}{2n } }p(i)^{q/2}}^{1/q} + \frac{1}{n^{1-1/q}}\paren{\sum_{p(i)< \frac{q}{2n}}p(i)}^{1/q}\\
        &\leq  \sqrt{\frac{q\|\p\|_{q/2}}{n}} + \frac{1}{\sqrt n}\paren{\sum_{p(i)< \frac{q}{2n}}p(i)}^{1/q}.
    \end{align*}
    It suffices to note that $\sum_{p(i)< \frac{q}{2n}}p(i) = o(1)$ as $n\to\infty$ to obtain
    \begin{equation*}
        \Ebb\|\hat \p_n-\p\|_q\leq (\Ebb \|\hat \p_n-\p\|_q^q)^{1/q} \lesssim \sqrt{\frac{q\|\p\|_{q/2}}{n}} + o\paren{\frac{1}{\sqrt n}}.
    \end{equation*}
    This ends the proof of the proposition.
\end{proof}

\section{Proofs of the results on high-probability bounds}\label{app:high_prob}

In this section, we provide high-probability bounds on the maximum deviation of the empirical mean for product distributions on $\{0,1\}$. We will need the following simple lemma on small empirical mean deviations of binomials.

\begin{lemma}\label{lemma:small_deviations}
    Let $p\in(0,\frac{1}{2}]$ and $n\geq \frac{4}{p}$. Then, for $\frac{1}{\sqrt{np}}\leq \delta \leq 1$,
    \begin{equation*}
        \Pbb_{Y\sim\Bcal(n,p)}\paren{\left|Y-np\right| \leq \delta\sqrt{np}} \asymp \delta.
    \end{equation*}
    In particular, for any $\delta\in(0,1]$, there is a universal constant $c>0$ for which
    \begin{equation*}
        \Pbb_{Y\sim\Bcal(n,p)}\paren{\left|Y-np\right|\lor 1 \leq c\delta\sqrt{np}} \leq \delta.
    \end{equation*}
\end{lemma}

\begin{proof}
    Let $k\in[np-\sqrt{np},np+\sqrt{np}]$ be an integer. We write $k=n(p+l)$ and note that $p+l\leq \frac{3}{4}$ because of the hypothesis. Then, using Stirling's formula
    \begin{equation*}
        \Pbb(Y=k) = \binom{n}{k} p^k (1-p)^{n-k}  \asymp \frac{1}{\sqrt n} \paren{\frac{p}{p+l}}^k\paren{\frac{1-p}{1-p-l}}^{n-k}.
    \end{equation*}
    But $ \paren{\frac{p}{p+l}}^k =\exp(-n(p+l)\ln(1+l/p))\asymp \exp(-np\ln(1+l/p)) \asymp e^{-nl}$. Also, $\paren{\frac{1-p}{1-p-l}}^{n-k}=\exp(n(1-p-l)\ln(1+l/(1-p-l)))\asymp e^{nl}$. Hence we obtained that there exist constants $b_1,b_2>0$ such that for any $k\in [np-\sqrt{np},np+\sqrt{np}]\cap \Nbb$, one has
    \begin{equation}\label{eq:similar_proba}
        \frac{b_1}{\sqrt n} \leq \Pbb(Y=k) \leq \frac{b_2}{\sqrt n}.
    \end{equation}
    By hypothesis, this set $[np-\sqrt{np},np+\sqrt{np}]\cap \Nbb$ contains at least $2\sqrt{np}-1\asymp 2\sqrt{np}$ elements.
    Now let $\frac{1}{\sqrt{np}}\leq \delta\leq 1$. Then, the interval $[np-\delta\sqrt{np},np+\delta\sqrt{np}]\cap \Nbb$ contains at least $\frac{1}{2}\delta \sqrt{np}$ elements. Together with Eq~\eqref{eq:similar_proba} this gives
    \begin{equation*}
        \Pbb(|Y-np|\leq \delta\sqrt{np})\asymp \delta \Pbb(|Y-np|\leq \sqrt{np}) \asymp \delta.
    \end{equation*}
    In the last inequality, we used Chernoff's bound.
\end{proof}

We prove the following bounds, which in particular include those stated in Proposition~\ref{prop:high_probability_bounds}.

\begin{proposition}
Let $\gamma\in(0,\frac{1}{2})$ and $\p\in[0,\frac{1}{2}]^\Nbb_{\downarrow 0}$ such that there exists $j\geq 1$ with $p(j)\geq \frac{\gamma}{2nj}$. Then, for some universal constants $a_1, a_2, a_3>0$,
\begin{equation*}
	\Pbb\paren{\|\hat \p_n - \p\|_\infty \geq a_1 \sup_{j\geq 1}\phi_{\frac{j}{\gamma},p(j)}(n)} \leq \gamma \quad \text{and} \quad \Pbb\paren{\|\hat \p_n - \p\|_\infty \geq a_2 \sup_{j\geq 1}\phi_{\frac{j}{\gamma},p(j)}(n)} \geq a_3\gamma
\end{equation*}
Also,
\begin{multline*}
	\Pbb\paren{\|\hat \p_n - \p\|_\infty \lor \frac{1}{n} \leq a_2 \sup_{j\geq 1}\phi_{\frac{j}{\ln 1/\gamma},p(j)}(n)} \leq \gamma \\
    \text{and} \quad \Pbb\paren{\|\hat \p_n - \p\|_\infty \leq a_1 \paren{\sup_{j\geq 1}\phi_{\frac{j}{\ln 1/\gamma},p(j)}(n) \lor \frac{1}{n}}} \geq \gamma^{3\ln\frac{1}{\gamma}}.
\end{multline*}

Let $\p\in[0,\frac{1}{2}]^\Nbb_{\downarrow 0}$ such that $p(j)\geq \frac{\gamma}{2nj}$ for all $j\geq 1$. Then,
\begin{equation*}
	\Pbb\paren{\sup_{j\geq 1} \hat p_n(j) \geq \frac{2}{n} }\leq \gamma \quad \text{and} \quad 								\Pbb\paren{\sup_{j\geq 1} \hat p_n(j) = \frac{1}{n} }\asymp 1\land n\sum_{j\geq 1}p(j).
\end{equation*}
\end{proposition}

\begin{proof}
We start with the first claim which gives upper bounds on $\|\hat \p_n-\p\|_\infty$. The only difference with the proofs for the bounds in expectation is that instead of $\varepsilon_{J,q}(n)$, we will use instead.
\begin{equation*}
	\varepsilon_{J/\gamma,q}(n) = \inf\set{\varepsilon\geq 0 : \Pbb_{Y\sim\Bcal(n,q)}\paren{\frac{Y}{n}\geq q+\varepsilon}\leq \frac{\gamma c_0}{2J}}.
\end{equation*}
Define $\varepsilon = \sup_{j\geq 1} \varepsilon_{j/\gamma,p(j)}(n;\gamma)$. 

We start with the case when there exists $j\geq 1$ for which $p(j)\leq \frac{\gamma}{2nj}$. Then, Proposition \ref{prop:combined_bound_eps} shows that $\varepsilon\geq 0$. We focus on the right tails. The proofs of the high-probability bounds now follow exactly the proof of Proposition \ref{prop:reduction}. Let $\tilde \varepsilon = (\varepsilon\land\frac{1}{4})\lor\frac{1}{n}$. The proof of Proposition \ref{prop:reduction} gives for any $i\geq 1$,
\begin{equation*}
D(p(i)+\tilde\varepsilon \parallel p(i))\geq \frac{\ln\frac{2i}{\gamma}}{nC},
\end{equation*}
then using Chernoff's bound and the union bound,
\begin{equation}\label{eq:useful_eq_hpb}
	\Pbb\paren{\sup_{j\geq 1}\hat p_n(j)-p(j)  \geq 2C\tilde\varepsilon}\leq \sum_{i\geq 1} \paren{\frac{\gamma}{2i}}^2 \leq \frac{\gamma}{2}.
\end{equation}
On the other hand, the proof of Proposition \ref{prop:reduction} also shows that this high-probability maximum deviation is tight up to constants. Indeed, it shows that for $i\geq 1$ such that $\varepsilon_i:=\varepsilon_{i/\gamma,p(i)}(n)\geq \frac{1}{2}\varepsilon$, one has for any $j\leq i$ with $p(j)\geq \frac{1}{n}$ or $p(j)\leq \frac{\varepsilon_i}{2}$,
\begin{equation*}
	\Pbb\paren{\hat p_n(j) \geq p(j) + \frac{\varepsilon_i}{2C}} \geq \frac{\gamma c_0^2}{2i}.
\end{equation*}
If $j\leq i$, $p(j)\leq \frac{1}{n}$ and $p(j)\geq \frac{\varepsilon_i}{2}$, the proof then showed that $\Pbb(\hat p_n(j) \geq p(j) +\frac{\varepsilon_i}{2})\geq \frac{1}{144}$. Hence, we obtain with the same proof
\begin{equation*}
	\Pbb\paren{\sup_{j\geq 1}\hat p_n(j)-p(j) \geq \frac{\varepsilon_i}{2C}} \geq \frac{1}{144} \lor (1-e^{-\gamma c_0^2/2}) \gtrsim \gamma.
\end{equation*}
We next turn to the left tails and show that these are dominated by the right tails. We first note that $D(p-2C\tilde\varepsilon\parallel p)\geq D(p+2C\tilde\varepsilon\parallel p)$. Indeed, with $f(x) = D(q+x\parallel q)$, we have $f'(x) = \ln\frac{q+x}{q} - \ln\frac{1-q-x}{1-q}$. Hence, if $x\geq 0$, $f'(x) + f'(-x) = \ln\paren{1-\frac{x^2}{q^2}} - \ln \paren{1-\frac{x^2}{(1-q)^2}} \leq 0$. Together with the fact that $f$ achieves its minimum at $x=0$ ends the proof of the claim. We then use the union bound together with Chernoff's bound as in Eq~\eqref{eq:useful_eq_hpb} to obtain
\begin{equation*}
	\Pbb\paren{\sup_{j\geq 1}p(j) - \hat p_n(j)  \geq 2C\tilde\varepsilon}\leq \sum_{i\geq 1}e^{-nD(p-2C\tilde\varepsilon\parallel p)} \leq \sum_{i\geq 1}e^{-nD(p+2C\tilde\varepsilon\parallel p)} \leq \frac{\gamma}{2}.
\end{equation*}

We next turn to lower bounds on $\|\hat\p_n-\p\|_\infty$. Again, we heavily use the proof of Proposition~\ref{prop:reduction}, but here we will use $\varepsilon_{j/\ln\frac{1}{\gamma},p(j)}(n)$ instead of $\varepsilon_{j,p(j)}(n)$. We define $\eta = \sup_{j\geq 2\ln\frac{1}{\gamma}/c_0^2} \varepsilon_{\frac{jc_0^2}{2\ln\frac{1}{\gamma}},p(j)}(n)$ and suppose for now that $\eta\geq 0$. Let $i\geq 2\ln\frac{1}{\gamma} /c_0^2$ be an integer such that $\eta_i:= \varepsilon_{\frac{i c_0^2}{2\ln\frac{1}{\gamma}},p(i)}(n) \geq \frac{\eta}{2}$. Then, the proof of Proposition~\ref{prop:reduction} shows that for all $j\leq i$, if either $p(j)\leq \frac{\eta_i}{2}$ or $p(j)\geq \frac{1}{n}$,
\begin{equation*}
    \Pbb\paren{\hat p_n(j)\geq p(j) + \frac{\eta_i}{2C}} \geq \frac{\ln\frac{1}{\gamma}}{i}.
\end{equation*}
On the other hand, if there is $j\leq i$ such that $\frac{\eta_i}{2} \leq p(j) \leq \frac{1}{n}$, then $\eta\leq 2\eta_i\leq \frac{4}{n}$. Then, we have directly that
\begin{equation*}
    \Pbb\paren{\|\hat \p_n-\p\|_\infty \lor \frac{1}{n} \geq \frac{\eta}{4}} = 1 \geq 1-\gamma.
\end{equation*}
As a result, in both cases, this gives
\begin{equation*}
    \Pbb\paren{\|\hat \p_n - \p\|_\infty\lor \frac{1}{n}\geq \frac{\eta}{4C}} \geq 1- \paren{1-\frac{\ln\frac{1}{\gamma}}{i}}^i\geq 1-\gamma.
\end{equation*}
We can then relate this equation to the quantities $\phi_{j/\ln\frac{1}{\gamma},p(j)}(n)$, using Proposition~\ref{prop:combined_bound_eps}, and the fact that a constant factor in $J$ only affects these quantities up to a constant factor. We then obtain directly for some universal constant $a_1$,
\begin{equation}\label{eq:desired_eq}
     \Pbb\paren{\|\hat \p_n - \p\|_\infty\lor \frac{1}{n}\geq a_1 \sup_{j\geq 2\ln\frac{1}{\gamma}/c_0^2} \phi_{j/\ln\frac{1}{\gamma},p(j)}(n)} \geq 1-\gamma.
\end{equation}
We next consider the case when $\eta<0$. By Proposition~\ref{prop:combined_bound_eps}, this implies that for any $j\geq 2\ln\frac{1}{\gamma}/c_0^2$, we have $p(j)\leq \frac{\ln\frac{1}{\gamma}}{c_0^2 nj}$. In particular, $\phi_{j/\ln\frac{1}{\gamma},p(j)}(n) \lesssim \frac{1}{n}$, hence obtaining Eq~\eqref{eq:desired_eq} in this case is immediate.

We now focus on the indices $i\leq 2\ln\frac{1}{\gamma}/c_0^2$. Note that if $p(i)\leq \frac{4}{n}$, then we again have $\phi_{j/\ln\frac{1}{\gamma},p(j)}(n) \lesssim \frac{1}{n}$. Without loss of generality, we can therefore suppose that $p(i)\geq \frac{4}{n}$. Lemma~\ref{lemma:small_deviations} implies that for some constant $c_1>0$, for any $j\leq i$,
\begin{equation*}
    \Pbb\paren{|\hat p_n(j) - p(j)|\lor \frac{1}{n} \leq c_1\gamma^{1/i} \sqrt{\frac{p(i)}{n}}} \leq \gamma^{1/i}.
\end{equation*}
As a result,
\begin{equation*}
    \Pbb\paren{\|\hat \p_n - \p\|_\infty \lor \frac{1}{n} \leq c'\gamma^{1/i} \sqrt{\frac{p(i)}{n}}} \leq \gamma.
\end{equation*}
In particular, this shows that for some universal constant $c_2>0$, we have
\begin{equation*}
    \Pbb\paren{\|\hat \p_n - \p\|_\infty \lor\frac{1}{n} \leq c_2 \sup_{j\leq 2\ln\frac{1}{\gamma}/c_0^2} \phi_{j/\ln\frac{1}{\gamma},p(j)}(n) } \leq \gamma.
\end{equation*}
Together with Eq~\eqref{eq:desired_eq}, we showed the desired bound for some constant $c_3>0$,
\begin{equation*}
    \Pbb\paren{\|\hat \p_n - \p\|_\infty \lor\frac{1}{n} \leq c_3 \sup_{j\geq 1} \phi_{j/\ln\frac{1}{\gamma},p(j)}(n) } \leq \gamma.
\end{equation*}

Next, we turn to the second inequality for the lower bound on $\|\hat\p_n-\p\|_\infty$. The proof of Proposition~\ref{prop:reduction} shows that with $\tilde\varepsilon' = \frac{1}{n}\lor \sup_{j\geq \ln\frac{1}{\gamma}} \varepsilon_{j/\ln\frac{1}{\gamma},p(j)}(n)$, for $j\geq \ln\frac{1}{\gamma}$, we have
\begin{equation*}
    D(p(j)+2C\tilde \varepsilon'\parallel p(j)) \geq \frac{2\ln(2j/\ln\frac{1}{\gamma})}{n}.
\end{equation*}
As a result, using Chernoff's bound and the fact that $D(p(j)-4C\tilde\varepsilon'\parallel p(j))\geq  D(p(j)+4C\tilde \varepsilon'\parallel p(j))$, we obtain
\begin{equation*}
    \Pbb(|\hat p_n(i) - p(i) | \geq 4C\tilde\varepsilon') \leq 2 e^{-nD(p(j)+4C\tilde \varepsilon'\parallel p(j)) } \leq \frac{\ln^4\frac{1}{\gamma}}{8i^4}.
\end{equation*}
Using the the inequality $\ln(1-x)\geq -2x$ for $x\in[0,\frac{1}{2}]$, we obtain
\begin{equation*}
    \Pbb\paren{\sup_{j\geq 1+\ln\frac{1}{\gamma}} |\hat p_n(j) - p(j) | \leq 4C\tilde\varepsilon'} \geq \exp\paren{-\sum_{j\geq 1+\ln\frac{1}{\gamma}}\frac{\ln^4\frac{1}{\gamma}}{4j^4}} \geq \exp\paren{\frac{\ln\frac{1}{\gamma}}{12}}=\gamma^{1/12}.
\end{equation*}
From Proposition~\ref{prop:combined_bound_eps}, we have that
\begin{equation*}
    \tilde\varepsilon \asymp \frac{1}{n}\lor \sup_{j\geq \ln\frac{1}{\gamma}} \phi_{j/\ln\frac{1}{\gamma},p(j)}(n).
\end{equation*}
Hence, it only remains to focus on the indices $i\leq 1+\ln\frac{1}{\gamma}$. First, note that in this case $\phi_{j/\ln\frac{1}{\gamma},p(j)}(n)\asymp \gamma^{1/i}\sqrt{\frac{p(i)}{n}}$. By Lemma~\ref{lemma:small_deviations}, for some constant $C_4>0$, we have
\begin{equation*}
    \Pbb\paren{\sup_{j\leq 1+\ln\frac{1}{\gamma}} |\hat p_n(j) - p(j) | \leq C_4\paren{\sup_{j\leq 1+\ln\frac{1}{\gamma}}\phi_{j/\ln\frac{1}{\gamma},p(j)}(n) \lor \frac{1}{n}}} \geq \gamma^{2\ln\frac{1}{\gamma}}.
\end{equation*}
Putting everything together yields the desired lower bound for some constant $C_5>0$ sufficiently large
\begin{equation*}
    \Pbb\paren{\|\hat \p_n - \p \|_\infty \leq C_5\paren{\sup_{j\geq 1}\phi_{j/\ln\frac{1}{\gamma},p(j)}(n) \lor \frac{1}{n}}} \geq \gamma^{3\ln\frac{1}{\gamma}}.
\end{equation*}

We now treat the case when $p(j)\leq \frac{\gamma}{2nj}$ for all $j\geq 1$, for which we use the proof of Proposition \ref{prop:small_p}. It directly gives with $\hat P_n:=\sup_{j\geq 1}\hat p_n(j)$ and $U(\p) = \sup_{j\geq 1} njp(j)\leq \frac{\gamma}{2}$ that
\begin{equation*}
	\Pbb\paren{\hat P_n \geq \frac{2}{n}} \leq \frac{\pi^2}{6} U(\p)^2 \leq \frac{\gamma}{2}.
\end{equation*}
On the other hand with $V(\p):=n \sum_{j\geq 1}p(j)$ which satisfies $V(\p)\land 1\geq U(\p)$, we have
\begin{equation*}
	c_4 V(\p)\land 1 \leq \Pbb\paren{ \hat P_n \geq \frac{1}{n}} \leq V(\p)\land 1,
\end{equation*}
for some constant $c_4>0$.
\end{proof}

\end{document}

%% file: shortcuts.tex
\newcommand{\trw}{\text{\small TRW}}
\newcommand{\maxcut}{\text{\small MAXCUT}}
\newcommand{\maxcsp}{\text{\small MAXCSP}}
\newcommand{\suol}{\text{SUOL}}
\newcommand{\wuol}{\text{WUOL}}
\newcommand{\crf}{\text{CRF}}
\newcommand{\sual}{\text{SUAL}}
\newcommand{\suil}{\text{SUIL}}
\newcommand{\fs}{\text{FS}}
\newcommand{\fmv}{{\text{FMV}}}
\newcommand{\smv}{{\text{SMV}}}
\newcommand{\wsmv}{{\text{WSMV}}}
\newcommand{\trwp}{\text{\small TRW}^\prime}
\newcommand{\alg}{\text{ALG}}
\newcommand{\rhos}{\rho^\star}
\newcommand{\brhos}{\brho^\star}
\newcommand{\bzero}{{\mathbf 0}}
\newcommand{\bs}{{\mathbf s}}
\newcommand{\bw}{{\mathbf w}}
\newcommand{\bws}{\bw^\star}
\newcommand{\ws}{w^\star}
\newcommand{\Prt}{{\mathsf {Part}}}
\newcommand{\Fs}{F^\star}

\newcommand{\Hs}{{\mathsf H} }

\newcommand{\hL}{\hat{L}}
\newcommand{\hU}{\hat{U}}
\newcommand{\hu}{\hat{u}}

\newcommand{\bu}{{\mathbf u}}
\newcommand{\ubf}{{\mathbf u}}
\newcommand{\hbu}{\hat{\bu}}

\newcommand{\primal}{\textbf{Primal}}
\newcommand{\dual}{\textbf{Dual}}

\newcommand{\Ptree}{{\sf P}^{\text{tree}}}
\newcommand{\bv}{{\mathbf v}}

\newcommand{\bq}{\boldsymbol q}

\newcommand{\rvM}{\text{M}}

\newcommand{\Acal}{\mathcal{A}}
\newcommand{\Bcal}{\mathcal{B}}
\newcommand{\Ccal}{\mathcal{C}}
\newcommand{\Dcal}{\mathcal{D}}
\newcommand{\Ecal}{\mathcal{E}}
\newcommand{\Fcal}{\mathcal{F}}
\newcommand{\Gcal}{\mathcal{G}}
\newcommand{\Hcal}{\mathcal{H}}
\newcommand{\Ical}{\mathcal{I}}
\newcommand{\Lcal}{\mathcal{L}}
\newcommand{\Ncal}{\mathcal{N}}
\newcommand{\Pcal}{\mathcal{P}}
\newcommand{\Scal}{\mathcal{S}}
\newcommand{\Tcal}{\mathcal{T}}
\newcommand{\Ucal}{\mathcal{U}}
\newcommand{\Vcal}{\mathcal{V}}
\newcommand{\Wcal}{\mathcal{W}}
\newcommand{\Xcal}{\mathcal{X}}
\newcommand{\Ycal}{\mathcal{Y}}
\newcommand{\Ocal}{\mathcal{O}}
\newcommand{\Qcal}{\mathcal{Q}}
\newcommand{\Rcal}{\mathcal{R}}

\newcommand{\brho}{\boldsymbol{\rho}}

\newcommand{\Cbb}{\mathbb{C}}
\newcommand{\Ebb}{\mathbb{E}}
\newcommand{\Nbb}{\mathbb{N}}
\newcommand{\Pbb}{\mathbb{P}}
\newcommand{\Qbb}{\mathbb{Q}}
\newcommand{\Rbb}{\mathbb{R}}
\newcommand{\Sbb}{\mathbb{S}}
\newcommand{\Xbb}{\mathbb{X}}
\newcommand{\Ybb}{\mathbb{Y}}
\newcommand{\Zbb}{\mathbb{Z}}

\newcommand{\Rbbp}{\Rbb_+}

\newcommand{\bX}{{\mathbf X}}
\newcommand{\bx}{{\boldsymbol x}}

\newcommand{\btheta}{\boldsymbol{\theta}}

\newcommand{\Pb}{\mathbb{P}}

\newcommand{\hPhi}{\widehat{\Phi}}

\newcommand{\Sigmah}{\widehat{\Sigma}}
\newcommand{\thetah}{\widehat{\theta}}

\newcommand{\indep}{\perp \!\!\! \perp}
\newcommand{\notindep}{\not\!\perp\!\!\!\perp}

\newcommand{\one}{\mathbbm{1}}
\newcommand{\1}{\mathbbm{1}}
\newcommand{\aprx}{\alpha}

\newcommand{\ST}{\Tcal(\Gcal)}
\newcommand{\x}{\mathsf{x}}
\newcommand{\y}{\mathsf{y}}
\newcommand{\Ybf}{\textbf{Y}}
\newcommand{\smiddle}[1]{\;\middle#1\;}

\definecolor{dark_red}{rgb}{0.2,0,0}
\newcommand{\detail}[1]{\textcolor{dark_red}{#1}}

\newcommand{\ds}[1]{{\color{red} #1}}
\newcommand{\rc}[1]{{\color{green} #1}}

\newcommand{\mb}[1]{\ensuremath{\boldsymbol{#1}}}

\newcommand{\metric}{\rho}

\newcommand{\paren}[1]{\left( #1 \right)}
\newcommand{\sqb}[1]{\left[ #1 \right]}
\newcommand{\floor}[1]{\lfloor #1 \rfloor}
\newcommand{\ceil}[1]{\lceil #1 \rceil}
\newcommand{\abs}[1]{\left|#1\right|}

%% file: main.bbl
\begin{thebibliography}{28}
\providecommand{\natexlab}[1]{#1}
\providecommand{\url}[1]{\texttt{#1}}
\expandafter\ifx\csname urlstyle\endcsname\relax
  \providecommand{\doi}[1]{doi: #1}\else
  \providecommand{\doi}{doi: \begingroup \urlstyle{rm}\Url}\fi

\bibitem[Balakrishnan and Wasserman(2019)]{balakrishnan2019hypothesis}
Sivaraman Balakrishnan and Larry Wasserman.
\newblock Hypothesis testing for densities and high-dimensional multinomials.
\newblock \emph{The Annals of Statistics}, 47\penalty0 (4):\penalty0
  1893--1927, 2019.

\bibitem[Bartl and Mendelson(2023)]{bartl2023variance}
Daniel Bartl and Shahar Mendelson.
\newblock On a variance dependent {D}voretzky-{K}iefer-{W}olfowitz inequality,
  2023.

\bibitem[Berend and Kontorovich(2013{\natexlab{a}})]{berend2013concentration}
Daniel Berend and Aryeh Kontorovich.
\newblock On the concentration of the missing mass.
\newblock \emph{Electronic Communications in Probability}, 18:\penalty0 1--7,
  2013{\natexlab{a}}.

\bibitem[Berend and Kontorovich(2013{\natexlab{b}})]{berend2013sharp}
Daniel Berend and Aryeh Kontorovich.
\newblock A sharp estimate of the binomial mean absolute deviation with
  applications.
\newblock \emph{Statistics \& Probability Letters}, 83\penalty0 (4):\penalty0
  1254--1259, 2013{\natexlab{b}}.

\bibitem[Boucheron et~al.(2013)Boucheron, Lugosi, and
  Massart]{boucheron2013concentration}
S~Boucheron, G~Lugosi, and P~Massart.
\newblock Concentration inequalities: A nonasymptotic theory of independence.
  univ. press, 2013.

\bibitem[Buldygin and Moskvichova(2013)]{buldygin2013sub}
V~Buldygin and K~Moskvichova.
\newblock The sub-gaussian norm of a binary random variable.
\newblock \emph{Theory of probability and mathematical statistics},
  86:\penalty0 33--49, 2013.

\bibitem[Catoni(2012)]{catoni2012challenging}
Olivier Catoni.
\newblock Challenging the empirical mean and empirical variance: a deviation
  study.
\newblock In \emph{Annales de l'IHP Probabilit{\'e}s et statistiques},
  volume~48, pages 1148--1185, 2012.

\bibitem[Cherapanamjeri et~al.(2019)Cherapanamjeri, Flammarion, and
  Bartlett]{cherapanamjeri2019fast}
Yeshwanth Cherapanamjeri, Nicolas Flammarion, and Peter~L Bartlett.
\newblock Fast mean estimation with sub-gaussian rates.
\newblock In \emph{Conference on Learning Theory}, pages 786--806. PMLR, 2019.

\bibitem[Cherapanamjeri et~al.(2022)Cherapanamjeri, Tripuraneni, Bartlett, and
  Jordan]{cherapanamjeri2022optimal}
Yeshwanth Cherapanamjeri, Nilesh Tripuraneni, Peter Bartlett, and Michael
  Jordan.
\newblock Optimal mean estimation without a variance.
\newblock In \emph{Conference on Learning Theory}, pages 356--357. PMLR, 2022.

\bibitem[Chhor and Carpentier(2020)]{chhor2020sharp}
Julien Chhor and Alexandra Carpentier.
\newblock Sharp local minimax rates for goodness-of-fit testing in multivariate
  binomial and poisson families and in multinomials.
\newblock \emph{arXiv preprint arXiv:2012.13766}, 2020.

\bibitem[Chhor and Carpentier(2021)]{chhor2021goodness}
Julien Chhor and Alexandra Carpentier.
\newblock Goodness-of-fit testing for h$\backslash$" older-continuous
  densities: Sharp local minimax rates.
\newblock \emph{arXiv preprint arXiv:2109.04346}, 2021.

\bibitem[Chhor et~al.(2022)Chhor, Mukherjee, and Sen]{chhor2022sparse}
Julien Chhor, Rajarshi Mukherjee, and Subhabrata Sen.
\newblock Sparse signal detection in heteroscedastic gaussian sequence models:
  Sharp minimax rates.
\newblock \emph{arXiv preprint arXiv:2211.08580}, 2022.

\bibitem[Cohen and Kontorovich(2022)]{cohen2022local}
Doron Cohen and Aryeh Kontorovich.
\newblock Local {G}livenko-{C}antelli.
\newblock \emph{arXiv preprint arXiv:2209.04054}, 2022.

\bibitem[Cohen and Kontorovich(2023)]{cohen2023open}
Doron Cohen and Aryeh Kontorovich.
\newblock Open problem: log(n) factor in "{L}ocal {G}livenko-{C}antelli.
\newblock \emph{COLT}, 2023.

\bibitem[Devroye et~al.(2016)Devroye, Lerasle, Lugosi, and
  Oliveira]{devroye2016sub}
Luc Devroye, Matthieu Lerasle, Gabor Lugosi, and Roberto~I Oliveira.
\newblock Sub-gaussian mean estimators.
\newblock 2016.

\bibitem[Diakonikolas et~al.(2020)Diakonikolas, Kane, and
  Pensia]{diakonikolas2020outlier}
Ilias Diakonikolas, Daniel~M Kane, and Ankit Pensia.
\newblock Outlier robust mean estimation with subgaussian rates via stability.
\newblock \emph{Advances in Neural Information Processing Systems},
  33:\penalty0 1830--1840, 2020.

\bibitem[Hopkins(2020)]{hopkins2020mean}
Samuel~B Hopkins.
\newblock Mean estimation with sub-gaussian rates in polynomial time.
\newblock 2020.

\bibitem[Kallenberg(1997)]{kallenberg1997foundations}
Olav Kallenberg.
\newblock \emph{Foundations of modern probability}, volume~2.
\newblock Springer, 1997.

\bibitem[Kearns and Saul(2013)]{kearns2013large}
Michael Kearns and Lawrence Saul.
\newblock Large deviation methods for approximate probabilistic inference.
\newblock \emph{arXiv preprint arXiv:1301.7392}, 2013.

\bibitem[Lee and Valiant(2022)]{lee2022optimal}
Jasper~C.H. Lee and Paul Valiant.
\newblock Optimal sub-gaussian mean estimation in $\mathbb{R}$.
\newblock In \emph{2021 IEEE 62nd Annual Symposium on Foundations of Computer
  Science (FOCS)}, pages 672--683, 2022.
\newblock \doi{10.1109/FOCS52979.2021.00071}.

\bibitem[Lugosi and Mendelson(2019{\natexlab{a}})]{lugosi2019mean}
G{\'a}bor Lugosi and Shahar Mendelson.
\newblock Mean estimation and regression under heavy-tailed distributions: A
  survey.
\newblock \emph{Foundations of Computational Mathematics}, 19\penalty0
  (5):\penalty0 1145--1190, 2019{\natexlab{a}}.

\bibitem[Lugosi and Mendelson(2019{\natexlab{b}})]{lugosi2019sub}
G{\'a}bor Lugosi and Shahar Mendelson.
\newblock Sub-gaussian estimators of the mean of a random vector.
\newblock 2019{\natexlab{b}}.

\bibitem[Lugosi and Mendelson(2021)]{lugosi2021robust}
Gabor Lugosi and Shahar Mendelson.
\newblock Robust multivariate mean estimation: the optimality of trimmed mean.
\newblock 2021.

\bibitem[Maillard(2021)]{maillard2021local}
Odalric-Ambrym Maillard.
\newblock Local {D}voretzky--{K}iefer--{W}olfowitz confidence bands.
\newblock \emph{Mathematical Methods of Statistics}, 30\penalty0
  (1-2):\penalty0 16--46, 2021.

\bibitem[Massart(1990)]{massart1990tight}
Pascal Massart.
\newblock The tight constant in the {D}voretzky-{K}iefer-{W}olfowitz
  inequality.
\newblock \emph{The annals of Probability}, pages 1269--1283, 1990.

\bibitem[Thomas(2018)]{thomas2018uniform}
Thomas.
\newblock Is uniform convergence faster for low-entropy distributions? in
  theoretical computer science stack exchange.
\newblock 2018.
\newblock URL \url{https://cstheory.stackexchange.com/questions/42009}.

\bibitem[Valiant and Valiant(2017)]{valiant2017automatic}
Gregory Valiant and Paul Valiant.
\newblock An automatic inequality prover and instance optimal identity testing.
\newblock \emph{SIAM Journal on Computing}, 46\penalty0 (1):\penalty0 429--455,
  2017.

\bibitem[Zhang and Zhou(2020)]{zhang2020non}
Anru~R Zhang and Yuchen Zhou.
\newblock On the non-asymptotic and sharp lower tail bounds of random
  variables.
\newblock \emph{Stat}, 9\penalty0 (1):\penalty0 e314, 2020.

\end{thebibliography}
